\documentclass{amsart}

\usepackage{amssymb}

\usepackage{graphicx,caption,subcaption}

\newtheorem{theorem}{Theorem}[section]
\newtheorem{lemma}[theorem]{Lemma}
\newtheorem{corollary}[theorem]{Corollary}

\theoremstyle{definition}

\theoremstyle{remark}
\newtheorem{remark}[theorem]{Remark}

\numberwithin{equation}{section}

\allowdisplaybreaks[4]

\begin{document}
\nocite*
\title{$L^{p}$ norms of the lattice point discrepancy}


 \author[L. Colzani]{Leonardo Colzani}
\address{Dipartimento di Matematica e Applicazioni,
   Universit\`a degli Studi di Milano-Bicocca,
   Via R. Cozzi 55, 20125 Milano,
   Italy}
\curraddr{}
\email{leonardo.colzani@unimib.it}
\thanks{}

\author[B. Gariboldi]{Bianca Gariboldi}
\address{Dipartimento di Ingegneria Gestionale, dell'Informazione e della Produzione,
  Universit\`a degli Studi di Bergamo,
  Viale Marconi 5, 24044 Dalmine (BG),
  Italy}
\curraddr{}
\email{biancamaria.gariboldi@unibg.it}
\thanks{}

\author[G. Gigante]{Giacomo Gigante}
\address{Dipartimento di Ingegneria Gestionale, dell'Informazione e della Produzione,
  Universit\`a degli Studi di Bergamo,
  Viale Marconi 5, 24044 Dalmine (BG),
  Italy}
\curraddr{}
\email{giacomo.gigante@unibg.it}
\thanks{}

\date{}

\dedicatory{}

\begin{abstract}
We estimate the $L^{p}$\ norms of the discrepancy
between the volume and the number of integer points in $r\Omega-x$, a dilated
by a factor $r$\ and translated by a vector $x$\ of a convex body $\Omega$\ in
$\mathbb{R}^{d}$ with smooth boundary  with strictly positive curvature,
\[
\left\{    
{\displaystyle\int_{\mathbb R}}{\displaystyle\int_{\mathbb{T}^{d}}}\left\vert \sum_{k\in\mathbb{Z}^{d}}\chi
_{r\Omega-x}(k)-r^{d}\left\vert \Omega\right\vert \right\vert ^{p}dxd\mu(r-R)
\right\}  ^{1/p},
\]
where $\mu$ is a Borel measure compactly supported on the positive real axis and $R\to+\infty$.
\end{abstract}

\maketitle

\section{Introduction}

\label{intro}

The discrepancy between the volume and the number of integer points in
$r\Omega-x$, a dilated by a factor $r$ and translated by a vector $x$ of
bounded domain $\Omega$ in $\mathbb{R}^{d}$, is
\[
\mathcal{D}\left(\Omega,r,x\right)  =%
{\displaystyle\sum_{k\in\mathbb{Z}^{d}}}
\chi_{r\Omega-x}(k)-r^{d}\left\vert \Omega\right\vert .
\]

Here $\chi_{r\Omega-x}(y)$ denotes the characteristic function of $r\Omega-x$
and $\left\vert \Omega\right\vert $ the measure of $\Omega$. 
A classical
problem is to estimate the size of $\mathcal{D}\left(  \Omega,r,x\right) $, as
$r\rightarrow+\infty$. For a survey see e.g. \cite{IKKN}, \cite{Kratzel} or \cite{Travaglini}.

By a classical result of D. G. Kendall,  the $L^{2}$ norm with respect to the translation variable $x$ of the discrepancy $\mathcal{D}\left(  \Omega,
r,x\right)  $ of an oval $\Omega$   is of the order of $r^{\left(  d-1\right)  /2}$. See \cite{Kendall} and what follows. For this reason we shall call
$r^{-\left(  d-1\right)  /2}\mathcal{D}\left(  \Omega,r,x\right)  $ the
normalized discrepancy. Our main result below is an estimate of
the fractal dimension of the set of values of the dilation variable $r$ where this normalized discrepancy may be large.

Throughout the paper, we shall assume that $\Omega$  is a convex body in $\mathbb R^d$ with smooth boundary  with strictly positive Gaussian curvature
such that the origin belongs to the interior of $\Omega$. 
We will also assume that $\mu$ is a positive Borel measure  with compact support contained in 
$\{0\leq r <+\infty\}$ and with Fourier transform $|\widehat\mu(\xi)|\le C(1+|\xi|)^{-\beta}$ for some $\beta\ge0$.
We recall that the Fourier transform of $\mu  $ is defined by
\begin{align*}
&  \widehat{\mu}\left(  \xi\right)  =%
{\displaystyle\int_{\mathbb{R}}}
\exp\left(  -2\pi i\xi r\right)  d\mu\left(  r\right),
\end{align*}
and that the Fourier dimension of a measure is the supremum of all $\delta$
such that there exists $C$ such that $\left\vert \widehat{\mu}\left(
\xi\right)  \right\vert \leq C\left\vert \xi\right\vert ^{-\delta/2}$.
See \cite[Section 4.4]{falconer} and \cite[Section 12.17]{PM}.

Also, for any $p\ge1$ and for any $R\ge 2$ we define
\[
I(d,\Omega,\mu,p,R)=\left\{  \int_{\mathbb{R}}  \int_{\mathbb{T}^{d}}| r^{-(d-1)/2}\mathcal D(\Omega,r,x)  | ^{p}dxd\mu(r-R)\right\}^{1/p},
\]
where the translated measure $d\mu\left(
r-R\right)  $ is defined by
\[%
{\displaystyle\int_{\mathbb{R}}}
f\left(  r\right)  d\mu\left(  r-R\right)  =%
{\displaystyle\int_{\mathbb{R}}}
f\left(  r+R\right)  d\mu\left(  r\right)  .
\]

 \begin{theorem} \label{thm_d=2} Let  $d=2$.
 
\noindent If $0\leq\beta<2/5$ then there exists a constant $C$ such that for every $R\ge 2$,
 \begin{align*}
   I(2,\Omega,\mu,p,R)
    \leq
 \begin{cases}
 C & \text{if }p<4+2\beta,\\
  C\log^{1/p}\left(  R\right)  & \text{if }p=4+2\beta.
 \end{cases}
 \end{align*}
 If $\beta=2/5$ then there exists a constant $C$ such that for every $R\ge 2$,
 \begin{align*}
  I(2,\Omega,\mu,p,R)
   \leq
 \begin{cases}
 C & \text{if }p<4+2\beta,\\
  C\log^{1/p+1/12}\left(  R\right)  & \text{if }p=4+2\beta.
 \end{cases}
 \end{align*}
 If $2/5<\beta<1/2$ then there exists a constant $C$ such that for every $R\ge 2$,
 \begin{align*}
 I(2,\Omega,\mu,p,R)
   \leq
 \begin{cases}
 C & \text{if }p<4+10\beta/(3+5\beta),\\
  C\log^{1/p}\left(  R\right)  & \text{if }p=4+10\beta/(3+5\beta).
 \end{cases}
 \end{align*}
  If $\beta=1/2$ then there exists a constant $C$ such that for every $R\ge 2$,
 \begin{align*}
  I(2,\Omega,\mu,p,R)
  \leq
 \begin{cases}
 C & \text{if }p<4+10/11,\\
  C\log^{1/p+1/9}\left(  R\right)  & \text{if }p=4+10/11.
 \end{cases}
 \end{align*}
  If $\beta>1/2$ then there exists a constant $C$ such that for every $R\ge 2$,
 \begin{align*}
  I(2,\Omega,\mu,p,R)
    \leq
 \begin{cases}
 C & \text{if }p<4+10/11,\\
  C\log^{1/p}\left(  R\right)  & \text{if }p=4+10/11.
 \end{cases}
 \end{align*}
 \end{theorem}

\begin{theorem}
\label{thm_d>2}
Let $d\ge3$. 

\noindent If $0\leq\beta<1$ then there exists a constant $C$ such that for every $R\ge 2$,
 \begin{align*}
 I(d,\Omega,\mu,p,R)
   \leq
 \begin{cases}
 C & \text{if }p<2(d-\beta)/(d-\beta-1),\\
  C\log^{1/p}\left(  R\right)  & \text{if }p=2(d-\beta)/(d-\beta-1).
 \end{cases}
 \end{align*}
If $\beta=1$ then there exists a constant $C$ such that for every $R\ge 2$,
 \begin{align*}
  I(d,\Omega,\mu,p,R)
      \leq
 \begin{cases}
 C & \text{if }p<2(d-1)/(d-2),\\
   C\log^{3/4}\left(  R\right)  & \text{if }p=2(d-1)/(d-2)\text{ and } d=3,\\
  C\log^{1/2}\left(  R\right)  & \text{if }p=2(d-1)/(d-2)\text{ and } d>3.
 \end{cases}
 \end{align*}
 If $\beta>1$ then there exists a constant $C$ such that for every  $R\ge 2$,
 \begin{align*}
   I(d,\Omega,\mu,p,R)
  \leq
 \begin{cases}
 C & \text{if }p<2(d-1)/(d-2),\\
   C\log^{1/2}\left(  R\right)  & \text{if }p=2(d-1)/(d-2)\text{ and } d=3,\\
  C\log^{1/p}\left(  R\right)  & \text{if }p=2(d-1)/(d-2)\text{ and } d>3.
 \end{cases}
 \end{align*}
 \end{theorem}

The case $d=2$ can be improved in the range $\beta>2/5$ when $\Omega$ is an ellipse $E$. More precisely
we have the following result.

 \begin{theorem} \label{thm_d=2_ellipse} Let $E$ be an ellipse in the plane.

\noindent If $0\leq\beta<1$ then there exists a constant $C$ such that for every $R\ge 2$,
 \begin{align*}
   I(2,E,\mu,p,R)
    \leq
 \begin{cases}
 C & \text{if }p<4+2\beta,\\
  C\log^{1/p}\left(  R\right)  & \text{if }\beta \neq 2/5 \text{ and } p=4+2\beta,\\
   C\log^{1/p+1/12}\left(  R\right)  & \text{if }\beta = 2/5 \text{ and } p=4+2\beta
 \end{cases}
 \end{align*}
 \noindent If $\beta=1$ then there exists a constant $C$ such that for every $R\ge 2$,
 \begin{align*}
  I(2,E,\mu,p,R)
   \leq
 \begin{cases}
 C & \text{if }p<6,\\
  C\log^{5/6}\left(  R\right)  & \text{if }p=6.
 \end{cases}
 \end{align*}
 If $\beta>1$ then there exists a constant $C$ such that for every $R\ge 2$,
 \begin{align*}
 I(2,E,\mu,p,R)
   \leq
 \begin{cases}
 C & \text{if }p<6,\\
  C\log^{2/3}\left(  R\right)  & \text{if }p=6.
 \end{cases}
 \end{align*}
 \end{theorem}

When the measure $\mu$ is the Dirac delta $\delta_0$ centered at $0$, then 
\[
I(d,\Omega,\delta_0,p,R)=
\left\{    \int_{\mathbb{T}^{d}}| R^{-(d-1)/2}\mathcal D(\Omega,R,x)  | ^{p}dx\right\}^{1/p}.
\]
In this case,  $|\widehat\delta_0(\xi)|=1$ so that $\beta=0$, and the above Theorems \ref{thm_d=2} and
\ref{thm_d>2} can be restated as
\begin{corollary}
\[
\left\{    \int_{\mathbb{T}^{d}}| R^{-(d-1)/2}\mathcal D(\Omega,R,x)  | ^{p}dx\right\}^{1/p}
\leq
\begin{cases}
 C & \text{if }p<2d/(d-1),\\
  C\log^{1/p}\left(  R\right)  & \text{if }p=2d/(d-1).
 \end{cases}
\]
\end{corollary}

This recovers recent results of M. Huxley \cite{Huxley3} for the case $d=2$
and L. Brandolini, L. Colzani, G. Gigante, G. Travaglini \cite{BCGT} for the general dimension $d$.

If $\mu $ is the uniformly distributed measure in the interval $\left\{
0<r<1\right\}  $, then 
\[
I(d,\Omega,\mu,p,R)=
\left\{ \int_{R}^{R+1}   \int_{\mathbb{T}^{d}}| r^{-(d-1)/2}\mathcal D(\Omega,r,x)  | ^{p}dxdr\right\}^{1/p}.
\]
The Fourier transform of the uniformly distributed measure in $\left\{
0<r<1\right\}  $ has decay $\beta=1$,
\[
\widehat{\mu}\left(  \xi\right)  =%
{\displaystyle\int_{0}^{1}}
\exp\left(  -2\pi i\xi r\right)  dr=\exp\left(  -\pi i\xi\right)
\dfrac{\sin\left(  \pi\xi\right)  }{\pi\xi}.
\]
On the other hand, if $\psi\left(  r\right)  $\ is a non negative smooth
function with integral one and support in $0\leq r\leq1$, one can consider a
smoothed average
\[
\left\{ \int_{\mathbb R} \int_{\mathbb{T}^{d}}\vert r^{-\left(
d-1\right)  /2}\mathcal{D}\left(  \Omega,r,x\right)  \vert ^{p}%
dx\psi\left(  r-R \right)  dr\right\}
^{1/p}.
\]
This smoothed average is equivalent to the uniform average over $\left\{
R<r<R+1\right\}  $, but the decay of the Fourier transform $\widehat{\psi
}  $ is faster than any power $\beta$. Hence for the
uniformly distributed measure in $\left\{  0<r<1\right\} $ Theorems \ref{thm_d=2} and
\ref{thm_d>2}
apply with the indices corresponding to $\beta>1$, and one obtains the following

\begin{corollary}
\begin{align*}
&  \left\{ \int_{R}^{R+1}   \int_{\mathbb{T}^{d}}| r^{-(d-1)/2}\mathcal D(\Omega,r,x)  | ^{p}dxdr\right\}^{1/p}\\
&  \leq\left\{
\begin{array}
[c]{ll}%
C & \text{if }d=2\text{ and }p<4+10/11\text{,}\\
C\log^{1/p}\left(  R\right)  & \text{if }d=2\text{ and }p=4+10/11\text{,}\\
C & \text{if }d\geq3\text{ and }p<2(d-1)/(d-2)\text{,}\\
C\log^{1/2}\left(  R\right)  & \text{if }d=3\text{ and }p=2(d-1)/(d-2)\text{,}\\
C\log^{1/p  }\left(  R\right)  &
\text{if }d>3\text{ and }p=2(  d-1)  /( d-2)\text{.}%
\end{array}
\right.
\end{align*}
\end{corollary}

As mentioned before, when $d=2$ and $\Omega=E$ (an ellipse) this result can be improved. By Theorem 
\ref{thm_d=2_ellipse},

\begin{corollary}
\[
 \left\{ \int_{R}^{R+1}   \int_{\mathbb{T}^{d}}| r^{-(d-1)/2}\mathcal D(E,r,x)  | ^{p}dxdr\right\}^{1/p}  \leq\left\{
\begin{array}
[c]{ll}%
C & \text{if }d=2\text{ and }p<6\text{,}\\
C\log^{2/3}\left(  R\right)  & \text{if }d=2\text{ and }p=6\text{.}\\%
\end{array}
\right.
\]
\end{corollary}

Observe that the range of indices for
which the  $L^{p}$ norm remains uniformly bounded with this choice of $\mu$
is larger than the range of indices in \cite{BCGT} and \cite{Huxley3} (those that we obtain
when $\mu=\delta_0$).

As an intermediate case between the two preceeding examples, one can consider a measure $d\mu\left(  r\right)=r^{-\alpha}\chi_{\left\{  0<r<1\right\}
}\left(  r\right)  dr$, with $0<\alpha<1$. In this case $\left\vert \widehat{\mu
}\left(  \xi\right)  \right\vert \leq C\left(  1+\left\vert \xi\right\vert
\right)  ^{\alpha-1}$, that is $\beta=1-\alpha$.
As a more sophisticated intermediate example, recall that a compactly supported probability measure is a Salem measure if its Fourier dimension $\gamma
=\sup\left\{  \delta:\ \left\vert \widehat{\mu}\left(  \zeta\right)
\right\vert \leq C\left(  1+\left\vert \zeta\right\vert \right)  ^{-\delta
/2}\ \right\}  $ is equal to the Hausdorff dimension of the support. Such
measures exist for every dimension $0<\gamma<1$. See \cite[Section 12.17]{PM}. 
The above theorems assert that the discrepancy cannot be too large in mean on the
supports of translations of these measures.

The techniques used to prove the above theorems also apply to the estimates of
mixed $L^{p}(L^2)$ norms of the discrepancy:
\[
\left\{  \int_{\mathbb{T}^{d}}\left(
{\displaystyle\int_{\mathbb{R}}}
\left\vert r^{-\left(  d-1\right)  /2}\mathcal{D}\left(  \Omega,r,x\right)
\right\vert ^{2}d\mu\left(  H^{-1}(r-R)\right)  \right)  ^{p/2}dx\right\}
^{1/p}.
\]

This has been done by the authors in \cite{CGG}. Here it suffices to remark that the set of $p$'s that give bounded mixed $L^p(L^2)$ norms is larger than the set of $p$'s that give bounded pure
$L^p$ norms. See also \cite{Gariboldi}.

We would like to thank Luca Brandolini and Giancarlo Travaglini for several discussions 
on this subject during the early stages of the preparation of this paper.

\section{Preliminary lemmas}

The proofs of the theorems will be splitted into a number of lemmas, some of them well known.
The starting point is the observation of D. G. Kendall that the discrepancy
$\mathcal{D}\left(  \Omega,r,x\right)  $ is a periodic function of the
translation, and it has a Fourier expansion with coefficients that are a
sampling of the Fourier transform of $\Omega$,
\[
\widehat{\chi}_{\Omega}\left(  \xi\right)  =%
{\displaystyle\int_{\Omega}}
\exp\left(  -2\pi i\xi x\right)  dx.
\]

\begin{lemma}
\label{Fourier} The number of integer points in $r\Omega-x$, a translated by a
vector $x\in\mathbb{R}^{d}$\ and dilated by a factor $r>0$\ of a domain
$\Omega$\ in the $d$ dimensional Euclidean space, is a periodic function of the
translation with Fourier expansion
\[%
{\displaystyle\sum_{k\in\mathbb{Z}^{d}}}
\chi_{r\Omega-x}(k)=%
{\displaystyle\sum_{n\in\mathbb{Z}^{d}}}
r^{d}\widehat{\chi}_{\Omega}\left(  rn\right)  \exp(2\pi inx).
\]
In particular,
\[
\mathcal{D}\left( \Omega,r,x\right)  =%
{\displaystyle\sum_{n\in\mathbb{Z}^{d}\setminus\left\{  0\right\}  }}
r^{d}\widehat{\chi}_{\Omega}\left(  rn\right)  \exp(2\pi inx).
\]

\end{lemma}

\begin{proof}
This is a particular case of the Poisson summation formula.
\end{proof}

\begin{remark}\label{r1}
We emphasize that the Fourier expansion of the discrepancy converges at least
in $L^{2}\left(  \mathbb{T}^{d}\right)  $, but we are not claiming that it
converges pointwise. Indeed, the discrepancy is discontinuous, hence the
associated Fourier expansion does not converge absolutely or uniformly. To
overcome this problem, one can introduce a mollified discrepancy. If the
domain $\Omega$\ is convex and contains the origin, then there exists
$\varepsilon>0$\ such that if $\varphi\left(  x\right)  $\ is a non negative
smooth radial function with support in $\left\{  \left\vert x\right\vert
\leq\varepsilon\right\}  $\ and with integral 1, and if $0<\delta\leq1$\ and
$r\geq1$, then
\[
\delta^{-d}\varphi(\delta^{-1}\cdot)\ast\chi_{(r-\delta)\Omega}(x)
\leq\chi_{r\Omega}(x)
\leq
\delta^{-d}\varphi(\delta^{-1}\cdot)\ast\chi_{(r+\delta)\Omega}(x)
\]
and therefore
\begin{gather*}
\left\vert \Omega\right\vert \left(  \left(  r-\delta\right)  ^{d}%
-r^{d}\right)  +\left(  r-\delta\right)  ^{d}%
{\displaystyle\sum_{n\in\mathbb{Z}^{d}\setminus\left\{  0\right\}  }}
\widehat{\varphi}\left(  \delta n\right)  \widehat{\chi}_{\Omega}\left(
\left(  r-\delta\right)  n\right)  \exp\left(  2\pi inx\right) \\
\leq%
\mathcal D\left(\Omega,r,x\right)\\
\leq\left\vert \Omega\right\vert \left(  \left(  r+\delta\right)  ^{d}%
-r^{d}\right)  +\left(  r+\delta\right)  ^{d}%
{\displaystyle\sum_{n\in\mathbb{Z}^{d}\setminus\left\{  0\right\}  }}
\widehat{\varphi}\left(  \delta n\right)  \widehat{\chi}_{\Omega}\left(
\left(  r+\delta\right)  n\right)  \exp\left(  2\pi inx\right)  .
\end{gather*}
One has $\left\vert \left(  r+\delta\right)  ^{d}-r^{d}\right\vert \leq
Cr^{d-1}\delta$, and one can work with the mollified discrepancy
defined by
\[
\mathcal D_{\delta}(\Omega,r,x)=
 r  ^{d}%
{\displaystyle\sum_{n\in\mathbb{Z}^{d}\setminus\left\{  0\right\}  }}
\widehat{\varphi}\left(  \delta n\right)  \widehat{\chi}_{\Omega}\left(
 rn\right)  \exp\left(  2\pi inx\right)  .
\]
Observe that since $\left\vert \widehat{\varphi
}\left(  \zeta\right)  \right\vert \leq C\left(  1+\left\vert \zeta\right\vert
\right)  ^{-\lambda}$ for every $\lambda>0$, the mollified Fourier
expansion has no problems of convergence.
\end{remark}

Let us recall that the support function of a convex set $\Omega\subset \mathbb R^d$ is defined by
$g(x)=\sup_{y\in\Omega}\left\{  x\cdot y\right\} $. 
When $\Omega$ is strictly convex with smooth boundary, the point $y$ that realizes the supremum in this definition  
is the point of $\partial \Omega$ where the outer normal
is parallel to $x$ (see Lemma \ref{Support} (2) and Figure \ref{F1} ).
For example, if $\Omega$ is the unit ball centered at the origin, then $g(x)=|x|$.

\begin{figure}[h!]
\begin{center}
\includegraphics[scale=0.8]{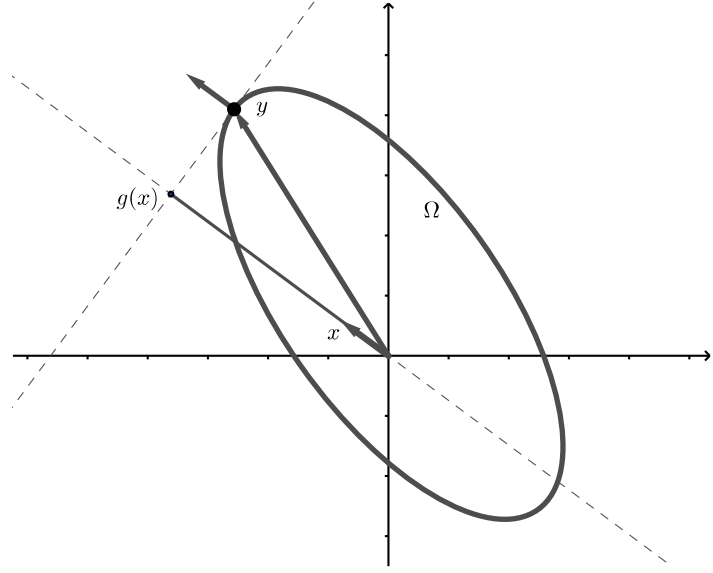}
\end{center}
\caption{The value of $g(x)$ when $|x|=1$. }\label{F1}
\end{figure}

\begin{lemma}
\label{Fourier Asymptotic copy(1)} Assume that $\Omega$\ is a convex body in
$\mathbb{R}^{d}$\ with smooth boundary with strictly positive Gaussian
curvature. Let $g\left(  x\right)  =\sup_{y\in\Omega
}\left\{  x\cdot y\right\}  $ be the support function of $\Omega$. Then, there exist functions $\left\{
a_{j}\left(  \xi\right)  \right\}  _{j=0}^{+\infty}$ and $\left\{
b_{j}\left(  \xi\right)  \right\}  _{j=0}^{+\infty}$ homogeneous of degree $0$
and smooth in $\mathbb{R}^{d}\setminus\left\{  0\right\}  $ such that the Fourier
transform of the characteristic function of $\Omega$ for $\left\vert
\xi\right\vert \rightarrow+\infty$ has the asymptotic expansion
\begin{align*}
&  \widehat{\chi}_{\Omega}\left(  \xi\right)  =%
{\displaystyle\int_{\Omega}}
\exp\left(  -2\pi i\xi\cdot x\right)  dx\\
&  =\exp\left(  -2\pi ig\left(  \xi\right)  \right)  \left\vert \xi\right\vert
^{-\left(  d+1\right)  /2}%
{\displaystyle\sum_{j=0}^{h}}
a_{j}\left(  \xi\right)  \left\vert \xi\right\vert ^{-j}\\
&+\exp\left(  2\pi
ig\left(  -\xi\right)  \right)  \left\vert \xi\right\vert ^{-\left(
d+1\right)  /2}%
{\displaystyle\sum_{j=0}^{h}}
b_{j}\left(  \xi\right)  \left\vert \xi\right\vert ^{-j}
  +\mathcal{O}\left(  \left\vert \xi\right\vert ^{-\left(  d+2h+3\right)
/2}\right)  .
\end{align*}
The functions $a_{j}\left(  \xi\right)  $ and $b_{j}\left(  \xi\right)  $
depend on a finite number of derivatives of a parame-trization of the boundary
of $\Omega$ at the points with outward unit normal $\pm\xi/\left\vert
\xi\right\vert $. In particular, $a_{0}\left(  \xi\right)  $ and $b_{0}\left(
\xi\right)  $ are, up to some absolute constants, equal to $K\left(  \pm
\xi\right)  ^{-1/2}$, with $K\left(  \pm\xi\right)  $\ the Gaussian curvature
of $\partial\Omega$\ at the points with outward unit normal $\pm\xi/\left\vert
\xi\right\vert $.
\end{lemma}

\begin{proof}
This is a classical result. See e.g. \cite{GGV}, 
\cite{Herz_decad}, \cite{Hlawka}, \cite{Stein}. Here, as an explicit example, we just recall that the
Fourier transform of a ball $\left\{  x\in\mathbb{R}^{d}:\left\vert
x\right\vert \leq R\right\}  $\ can be expressed in terms of a Bessel
function, and Bessel functions have simple asymptotic expansions in terms of
trigonometric functions,
\begin{align*}
&  \widehat{\chi}_{\left\{  \left\vert x\right\vert \leq R\right\}  }\left(
\xi\right)  =R^{d}\widehat{\chi}_{\left\{  \left\vert x\right\vert
\leq1\right\}  }\left(  R\xi\right)  =R^{d}\left\vert R\xi\right\vert
^{-d/2}J_{d/2}\left(  2\pi\left\vert R\xi\right\vert \right) \\
& =\pi^{-1}R^{(d-1)/2}\left\vert \xi\right\vert ^{-\left(  d+1\right)
/2}\cos\left(  2\pi R\left\vert \xi\right\vert -\left(  d+1\right)
\pi/4\right) \\
&  -2^{-4}\pi^{-2}\left(  d^{2}-1\right)  R^{(d-3)/2}\left\vert \xi\right\vert
^{-\left(  d+3\right)  /2}\sin\left(  2\pi R\left\vert \xi\right\vert -\left(
d+1\right)  \pi/4\right)  +...\\
&+O(R^{(d-2h-3)/2}|\xi|^{-(d+2h+3)/2}).
\end{align*}
See \cite{SW}. More generally, also the Fourier transform of an ellipsoid, that is an affine
image of a ball, can be expressed in terms of Bessel functions.
\end{proof}

By the above lemma, the normalized discrepancy has a Fourier expansion of the form
\begin{align*}
&\sum_{j=0}^h r^{-j}%
{\displaystyle\sum_{n\in\mathbb{Z}^{d}\setminus\left\{  0\right\}  }}
a_{j}\left(  n\right)  \left\vert n\right\vert ^{-(d+1)/2-j}\exp\left(  -2\pi
ig\left(  n\right)  r\right)  \exp\left(  2\pi inx\right) \\
&  +\sum_{j=0}^hr^{-j}%
{\displaystyle\sum_{n\in\mathbb{Z}^{d}\setminus\left\{  0\right\}  }}
b_{j}\left(  n\right)  \left\vert n\right\vert ^{-(d+1)/2-j}\exp\left(  2\pi
ig\left(  -n\right)  r\right)  \exp\left(  2\pi inx\right)+\ldots
\end{align*}
By replacing $(d+1)/2$ with a complex variable $z$ one obtains an analytic function of this complex 
variable with values in $L^p$ spaces. This will allow us to estimate the norm of this discrepancy via complex interpolation.

\begin{lemma}
\label{Asymptotic Discrepancy} Assume that $\Omega$\ is a convex body in
$\mathbb{R}^{d}$\ with smooth boundary with strictly positive Gaussian
curvature. Let $z$ be a complex parameter, and for every $h=0,1,2,...$ and
$r\geq1$, with the notation of the previous lemmas, let define the function 
$\Phi\left( \delta, z,r,x\right)  $ via the Fourier expansions
\begin{align*}
&  \Phi\left(  \delta,z,r,x\right)  \\
&=\sum_{j=0}^h r^{-j}%
{\displaystyle\sum_{n\in\mathbb{Z}^{d}\setminus\left\{  0\right\}  }}
\widehat\varphi(\delta n)a_{j}\left(  n\right)  \left\vert n\right\vert ^{-z-j}\exp\left(  -2\pi
ig\left(  n\right)  r\right)  \exp\left(  2\pi inx\right) \\
&  +\sum_{j=0}^hr^{-j}%
{\displaystyle\sum_{n\in\mathbb{Z}^{d}\setminus\left\{  0\right\}  }}
\widehat\varphi(\delta n)b_{j}\left(  n\right)  \left\vert n\right\vert ^{-z-j}\exp\left(  2\pi
ig\left(  -n\right)  r\right)  \exp\left(  2\pi inx\right)  .
\end{align*}
The convergence of the above series is absolute and uniform. With $z=(d+1)/2$, define
\[
\mathcal{R}_{h}\left( \delta, r,x\right)  =r^{-\left(  d-1\right)  /2}\mathcal{D}_\delta%
\left(  \Omega,r,x\right)  -%
\Phi\left( \delta, \left(  d+1\right)  /2,r,x\right)  .
\]
Then, if $h>\left(  d-3\right)  /2$ there exists $C$ such that for every $\delta>0$ and $r\geq1$,
\[
\left\vert \mathcal{R}_{h}\left( \delta, r,x\right)  \right\vert \leq Cr^{-h-1}.
\]

\end{lemma}

\begin{proof}
This is a consequence of the previous lemmas. 
\end{proof}

\begin{lemma} \label{lm_6_integral_ellipsoid}
Let $N$ be a positive integer, and $\mu$ a positive measure with compact support in the positive real axis and with Fourier transform satisfying $|\widehat\mu(\xi)|\leq (1+|\xi|)^{-\beta}$  for some $\beta\geq0$.
Then for every $\lambda>0$ and for every $z\in\mathbb C$ there exists $C>0$ such that for every $\delta>0$ and for every $1<R<+\infty$,
\begin{align*}
& \int_{\mathbb{R}} \int_{\mathbb{T}^{d}}  |\Phi(\delta,z,r,x)|^{2N}dxd\mu(r-R)\leq C\int_{\mathbb{R}^{d}} (1+\delta|k|)^{-\lambda}\\
& \times\int\limits_{\substack{m_1,\ldots,m_{N}\in\mathbb{R}^{d}\\ |m_1|, \ldots,|m_{N}| >1\\ m_1+\dots+m_{N}=k }} (1+\delta|m_1|)^{-\lambda}\dots(1+\delta|m_N|)^{-\lambda}|m_1|^{-\mathrm{Re}(z)}\dots|m_{N}|^{-\mathrm{Re}(z)}\\
& \times\int\limits_{\substack{n_1,\ldots,n_{N}\in\mathbb{R}^{d}\\ |n_1|, \ldots,|n_{N}| >1\\ n_1+\dots+n_{N}=k }}(1+\delta|n_1|)^{-\lambda}\dots(1+\delta|n_N|)^{-\lambda} |n_1|^{-\mathrm{Re}(z)}\dots|n_{N}|^{-\mathrm{Re}(z)} \\
&\times \left(1+| g(m_1)+\dots+g(m_{N})-g(n_1)-\dots-g(n_{N})|\right)^{-\beta}\\
&\times d\sigma(n_1,\dots,n_{N}) d\sigma(m_1,\dots,m_{N}) dk.
\end{align*}
The inner integrals are with respect to the surface measure on the $(N-1)d$ dimensional variety of $N$ points with sum $k$. In other words, they are essentially Lebesgue integrals with respect to the first $N-1$ variables, replacing the last variables $ m_N$ and $n_N$ respectively with $k-m_1-\dots-m_{N-1}$
and $k-n_1-\dots-n_{N-1}$.
The above final expression is a decreasing function of $\mathrm{Re}(z)$.

\end{lemma}

\begin{proof} The function $\Phi(\delta,z,r,x)$ is a sum of terms of the form
\[
\Theta_j\left(  \delta,z,r,x\right)  =r^{-j}%
{\displaystyle\sum_{n\in\mathbb{Z}^{d}\setminus\left\{  0\right\}  }}
\widehat\varphi(\delta n)c(n)\left\vert n\right\vert ^{-z-j}\exp\left(  \mp2\pi ig\left(   \pm n\right)
r\right)  \exp\left(  2\pi inx\right)  .
\]
The expression $\exp\left(  \mp2\pi ig\left(   \pm n\right)
r\right)$ has to be intended either as $\exp\left(  -2\pi ig\left(   n\right)
r\right)$ or $\exp\left(  +2\pi ig\left(  - n\right)
r\right)$.
Assume that the first exponential is $\exp(-2\pi ig\left(  n\right)r)$.
It follows that for every positive integer $N$,
\begin{align*}
& \left(\Theta_j(\delta,z,r,x)\right)^{N}\\
=&  r^{-jN}\sum_{k\in\mathbb{Z}^{d}} \sum_{\substack{n_1,\dots n_N\neq 0\\ n_1+\dots+n_N=k}} \widehat\varphi(\delta n_1)\dots\widehat\varphi(\delta n_N)c(n_1)\dots c(n_N) |n_1|^{-z-j}\dots |n_N|^{-z-j}\\
& \times e^{-2\pi i\left(g(n_1)+\dots+g(n_N)\right)r}e^{2\pi ikx}.
\end{align*}
For a proof, just observe that since $\widehat{\varphi}(\xi)$ has a fast decay at infinity, all series involved are absolutely convergent, and one can freely expand the $N$-th power and rearrange the terms. 
Then, by Parseval equality,
\begin{align*}
& \int_{\mathbb{T}^{d}} |\Theta_j(\delta,z,r,x)|^{2N}dx\\
& = r^{-2jN}\sum_{k\in\mathbb{Z}^{d}} \Bigg|  \sum_{\substack{n_1,\dots n_N\neq 0\\ n_1+\dots+n_N=k}} \widehat\varphi(\delta n_1)\dots\widehat\varphi(\delta n_N)c(n_1)\dots c(n_N) |n_1|^{-z-j}\dots |n_N|^{-z-j}\\
& \times e^{-2\pi i\left(g(n_1)+\dots+g(n_N)\right)r}\Bigg|^{2}.
\end{align*}
Expanding the square and integrating in the variable $r$, one obtains
\begin{align*}
&  \int_{\mathbb{R}} \int_{\mathbb{T}^{d}} |\Theta_j(\delta,z,r,x)|^{2N}dxd\mu(r-R)\\
&=\sum_{k\in\mathbb{Z}^{d}}  \sum_{\substack{m_1,\dots m_N\neq 0\\ m_1+\dots+m_N=k}} \widehat\varphi(\delta m_1)\dots\widehat\varphi(\delta m_N)c(m_1)\dots c(m_N) |m_1|^{-z-j}\dots |m_N|^{-z-j}\\
&\times  \sum_{\substack{n_1,\dots n_N\neq 0\\ n_1+\dots+n_N=k}} \widehat\varphi(\delta n_1)\dots\widehat\varphi(\delta n_N)\overline{c(n_1)}\dots \overline{c(n_N)} |n_1|^{-\overline z-j}\dots |n_N|^{-\overline z-j}\\
&  \times \int_{\mathbb{R}} e^{-2\pi i\left(g(m_1)+\dots+g(m_N)-g(n_1)-\dots-g(n_N)\right)r} r^{-2jN}d\mu(r-R).
\end{align*}
The last integral is the Fourier transform of the measure $r^{-2jN}d\mu(r-R)$ and
it is easy to see that if $R>1$, so that the singularity of $r^{-2jN}$ is far from the support of $d\mu(r-R)$, it satisfies the same estimates as the Fourier transform of $d\mu(r)$. Indeed
\begin{align*}
&\int_{\mathbb R} e^{-2\pi i\xi r}r^{-2jN}d\mu(r-R)=e^{-2\pi i\xi R}\int_{\mathbb R} e^{-2\pi i\xi r}(r+R)^{-2jN}d\mu(r)\\
&=e^{-2\pi i\xi R}\int_{\mathbb R}\widehat{d\mu}(\xi-s) q(s)ds,
\end{align*}
where $q(s)$ is the Fourier transform of a smooth extension outside the support of $d\mu$ of the function $(r+R)^{-2jN}$.
Hence,
\begin{align*}
&\left|\int_{\mathbb{R}} e^{-2\pi i\left(g(m_1)+\dots+g(m_N)-g(n_1)-\dots-g(n_N)\right)r} r^{-2jN}d\mu(r-R)\right|\\
&\leq
C\left(1+\left|g(m_1)+\dots+g(m_N)-g(n_1)-\dots-g(n_N)\right|\right)^{-\beta}.
\end{align*}
The function $\varphi(x)$ is smooth, so that $|\widehat{\varphi}(\xi)|\leq C\left(1+|\xi|\right)^{-\lambda}$  for every $\lambda>0$. 
Hence for every $j$ the above quantity is dominated up to a constant by
\begin{align*}
&\sum_{k\in\mathbb{Z}^{d}}  \sum_{\substack{m_1,\dots m_N\neq 0\\ m_1+\dots+m_N=k}} 
\left(1+|\delta m_1|\right)^{-\lambda}\dots (1+|\delta m_N|)^{-\lambda}
|m_1|^{-\mathrm{Re}(z)}\dots |m_N|^{-\mathrm{Re}(z)}\\
&\times  \sum_{\substack{n_1,\dots n_N\neq 0\\ n_1+\dots+n_N=k}} 
(1+|\delta n_1|)^{-\lambda}\dots(1+|\delta n_N|)^{-\lambda}  
|n_1|^{-\mathrm{Re}(z)}\dots |n_N|^{-\mathrm{Re}(z)}\\
&  \times\left(1+\left|g(m_1)+\dots+g(m_N)-g(n_1)-\dots-g(n_N)\right|\right)^{-\beta} .
\end{align*}
In this formula there is no cutoff in the variable $k$. In order to obtain a cutoff in $k$, observe that, if $m_1+\dots+m_N=k$, then
\begin{align*}
&  (1+\delta|m_1|)^{-\sigma}\dots(1+\delta |m_{N}|)^{-\sigma} \\
=& \left(1+\delta(|m_1|+\dots+ |m_N|) + \delta^{2}(|m_1| |m_2|+\dots)+\dots\right)^{-\sigma}\\
&\leq \left(1+\delta(|m_1|+\dots+ |m_N|)\right)^{-\sigma} \leq \left(1+\delta |m_1+\dots+m_N| \right)^{-\sigma} = \left(1+\delta |k|\right)^{-\sigma}.
\end{align*}
In particular, some of the cutoff functions $(1+\delta|m_1|)^{-\sigma}\dots (1+\delta |m_N|)^{-\sigma}$ can be replaced with $(1+\delta |k|)^{-\sigma}$. 
Finally, in the above formulas one can replace the sums with integrals. Indeed, there exist positive constants $A$ and $B$ such that for every integer point $m\neq0$ and every $x\in Q$, the cube centered at the
origin with sides parallel to the axes and of length one,
\[
A |m| \leq |m+x| \leq B |m|.
\]
This implies that the function $|m+x|^{-\mathrm{Re}(z)}$ is slowly varying in the cube $Q$. 
Moreover, also the function
$
\left(1+| g(m+x)+\dots|\right)^{-\beta}
$
is slowly varying. 
Hence, one can replace a sum over $m$ with an integral over the union of cubes $m+Q$,
\begin{align*}
&\sum_{k\in\mathbb{Z}^{d}}  \left(1+|\delta k|\right)^{-\lambda}\\
&\times\sum_{\substack{m_1,\dots m_N\neq 0\\ m_1+\dots+m_N=k}} 
\left(1+|\delta m_1|\right)^{-\lambda}\dots (1+|\delta m_N|)^{-\lambda}
|m_1|^{-\mathrm{Re}(z)}\dots |m_N|^{-\mathrm{Re}(z)}\\
&\times  \sum_{\substack{n_1,\dots n_N\neq 0\\ n_1+\dots+n_N=k}} 
(1+|\delta n_1|)^{-\lambda}\dots(1+|\delta n_N|)^{-\lambda}  
|n_1|^{-\mathrm{Re}(z)}\dots |n_N|^{-\mathrm{Re}(z)}\\
&  \times\left(1+\left|g(m_1)+\dots+g(m_N)-g(n_1)-\dots-g(n_N)\right|\right)^{-\beta} 
\\
&\leq C\int_{\mathbb{R}^{d}} (1+\delta|k|)^{-\lambda}\\
&\times\int\limits_{\substack{m_1,\ldots,m_{N}\in\mathbb{R}^{d}\\ |m_1|, \ldots,|m_{N}| >1/2\\ m_1+\dots+m_{N}=k }} (1+\delta|m_1|)^{-\lambda}\dots(1+\delta|m_N|)^{-\lambda}|m_1|^{-\mathrm{Re}(z)}\dots|m_{N}|^{-\mathrm{Re}(z)}\\
& \times\int\limits_{\substack{n_1,\ldots,n_{N}\in\mathbb{R}^{d}\\ |n_1|, \ldots,|n_{N}| >1/2\\ n_1+\dots+n_{N}=k }}(1+\delta|n_1|)^{-\lambda}\dots(1+\delta|n_N|)^{-\lambda} |n_1|^{-\mathrm{Re}(z)}\dots|n_{N}|^{-\mathrm{Re}(z)} \\
&\times \left(1+| g(m_1)+\dots+g(m_{N})-g(n_1)-\dots-g(n_{N})|\right)^{-\beta}\\
&\times d\sigma(n_1,\dots,n_{N}) d\sigma(m_1,\dots,m_{N}) dk.
\end{align*}
Finally, with a change of variables one can transform the domain of integration $\left\{|x|>1/2\right\}$ into $\left\{|y|>1\right\}$, and the thesis follows immediately. 
Indeed, if $|x|>1$ then $|x|^{-\mathrm{Re}(z)}$ decreases as $\mathrm{Re}(z)$ increases.
\end{proof}

\section{The case of a generic convex set with non vanishing curvature}
 This section contains the proof of Theorems \ref{thm_d=2} and \ref{thm_d>2}.  
 
\begin{lemma}
\label{Support} Let $g\left(  x\right)  =\sup_{y\in\Omega}\left\{  x\cdot
y\right\}  $ be the support function of a convex set $\Omega$ which contains the
origin, and with a smooth boundary with strictly positive Gaussian curvature.

(1) The support function is convex, homogeneous of degree one, positive and
smooth away from the origin, and it is equivalent to the Euclidean norm, that
is there exist $0<A<B$ such that for every $x$,
\[
A\left\vert x\right\vert \leq g\left(  x\right)  \leq B\left\vert x\right\vert
.
\]

(2) For every $x\in\mathbb R^d\setminus\{0\}$ there exist a unique point $z(x)\in\partial\Omega$
such that $x\cdot z(x)=g(x)$. In particular, $z(x)$ is the unique point in $\partial\Omega$ with outer normal $x$.
Furthermore $z(x)$ is homogeneous of degree $0$ in the variable $x$.

(3) There exist two positive constants  $c$ and $C$ such that for every $\vartheta,\omega$ with $\left|\vartheta\right|=\left|\omega\right|=1$
\[
c\left(1-\omega\cdot\vartheta\right)\le z(\vartheta)\cdot\vartheta-z\left(\omega\right)\cdot\vartheta\le C\left(1-\omega\cdot\vartheta\right)
\]

\end{lemma}

\begin{proof}
(1) is trivial and (2) follows from the smoothness and positive curvature of the boundary of $\Omega$.
The estimate in (3) can be written as 
\[
c\left(\vartheta-\omega\right)\cdot\vartheta\le (z(\vartheta)-z\left(\omega\right))\cdot\vartheta\le C\left(\vartheta-\omega\right)\cdot\vartheta.
\]
A compactness argument implies that there exist two radii $c\leq C$ such that
at every point of $\partial\Omega$, $\Omega$ contains the ball with radius $c$
tangent to $\partial\Omega$ and is contained in the ball with radius $C$ tangent to $\partial\Omega$, and (3) follows.

\end{proof}

\begin{lemma}
In the hypotheses of the above lemma, assume that $\beta\geq
0$, $\delta,\gamma>0$  and $Y\geq1$.  Set
\begin{equation*}
\tau=\tau(\gamma,\beta)   =\min\left\{  1,\gamma,\beta\right\} \qquad
\sigma =\sigma(\gamma,\beta)  =\left\{
\begin{array}
[c]{ll}%
1 & \text{if }\beta=1\text{ and }\gamma\geq1,\\
1 & \text{if }\beta=\gamma\leq1,\\
0 & \text{else.}%
\end{array}
\right.
\end{equation*}
Then, for every $\vartheta$ with
$\left\vert \vartheta\right\vert =1$ and every $k\neq0$, if we call $\Delta=g\left(  \vartheta\right)  -\nabla g\left(  k\right)  \cdot\vartheta
$,
\begin{equation*}
 \int_{0}^{\delta Y}\rho^{\gamma-1}\left(  1+\left\vert g\left(
\rho\vartheta\right)  +g\left(  k-\rho\vartheta\right)  -Y\right\vert \right)
^{-\beta}d\rho
  \leq CY^{\gamma}\left(  1+Y\Delta\right)  ^{-\tau}\log^{\sigma}\left(
2+Y\Delta\right).
\end{equation*}

\end{lemma}

\begin{remark}
Here and in the rest of the paper we will use the Kronecker delta notation
\[
\delta_{x,y}=\begin{cases}
1& \text{ if } x=y,\\
0 & \text{ if } x\neq y.
\end{cases}
\]
\end{remark}

\begin{proof}
The case $\beta=0$ is trivial. When $\beta>0$, then $\left(  1+Y\Delta\right)  ^{-\tau}\log^{\sigma}\left(
2+Y\Delta\right)$ is the gain with respect to the trivial estimate 
\begin{equation*}
 \int_{0}^{\delta Y}\rho^{\gamma-1}d\rho
  = C Y^{\gamma}.
\end{equation*}
Fix $\vartheta$ and $k$ and define%
\begin{align*}
F\left(  \rho\right)   &  =g\left(  \rho\vartheta\right)  +g\left(
k-\rho\vartheta\right)  ,\\
F^{\prime}\left(  \rho\right)   &  =g\left(  \vartheta\right)  -\nabla
g\left(  k-\rho\vartheta\right)  \cdot\vartheta,\\
F^{\prime\prime}\left(  \rho\right)   &  =\vartheta^{t}\cdot\nabla^{2}g\left(
k-\rho\vartheta\right)  \cdot\vartheta.
\end{align*}

The quadratic form $\vartheta^{t}\cdot\nabla^{2}g\left(  x\right)
\cdot\vartheta$ is positive semidefinite, with a $0$ eigenvalue in the
direction given by $\nabla g\left(  x\right)  $, and strictly positive
eigenvalues in the other directions. See \cite{Schneider}. Let $\lambda>0$ be the minimum
of all the $d-1$ positive eigenvalues of $\vartheta^{t}\cdot\nabla^{2}g\left(
x\right)  \cdot\vartheta$ over the sphere $\left\{  \left\vert x\right\vert
=1\right\}  $. Splitting $\vartheta=\vartheta_{0}+\vartheta_{1}$ with
$\vartheta_{0}$ parallel to $\nabla g\left(  k-\rho\vartheta\right)  $ and
$\vartheta_{1}$ orthogonal to $\nabla g\left(  k-\rho\vartheta\right)  $, and
since $\nabla^{2}g\left(  x\right)  $ is homogeneous of degree $-1$, we
therefore have
\begin{align*}
&  \vartheta^{t}\cdot\nabla^{2}g\left(  k-\rho\vartheta\right)  \cdot
\vartheta\\
&  =\vartheta_{1}^{t}\cdot\nabla^{2}g\left(  k-\rho\vartheta\right)
\cdot\vartheta_{1}\\
&  =\frac{1}{\left\vert k-\rho\vartheta\right\vert }\vartheta_{1}^{t}%
\cdot\nabla^{2}g\left(  \frac{k-\rho\vartheta}{\left\vert k-\rho
\vartheta\right\vert }\right)  \cdot\vartheta_{1}\\
&  \geq\frac{\lambda\left\vert \vartheta_{1}\right\vert ^{2}}{\left\vert
k-\rho\vartheta\right\vert }.
\end{align*}

In particular, $F^{\prime\prime}\left(  \rho\right)  \geq0$, so that
$F^{\prime}\left(  \rho\right)  $ is increasing and if $\rho\geq0$,
\[
F^{\prime}\left(  \rho\right)  \geq F^{\prime}\left(  0\right)  =g\left(
\vartheta\right)  -\nabla g\left(  k\right)  \cdot\vartheta.
\]

If $z\left(  x\right)  $ is the unique point in the boundary of $\Omega$ such
that $x\cdot z\left(  x\right)  =g\left(  x\right) $, as described in Lemma \ref{Support}, then
\[
\nabla g\left(  x\right)  =\nabla\left(  z\left(  x\right)  \cdot x\right)
=\nabla z\left(  x\right)  \cdot x+z\left(  x\right)  =z\left(  x\right)  .
\]

The last equality follows from Euler's formula, since $z\left(  x\right)  $ is
homogeneous of degree $0$. Hence,%
\[
F^{\prime}\left(  0\right)  =g\left(  \vartheta\right)  -\nabla g\left(
k\right)  \cdot\vartheta=g\left(  \vartheta\right)  -z\left(  k\right)
\cdot\vartheta\geq0.
\]

This follows by the definition of $g\left(  \vartheta\right)  $ as the maximum
of $y\cdot\vartheta$ for $y\in\Omega$.

Thus, $F\left(  \rho\right)  $ is an increasing function. If $Y\geq F\left(
0\right)  $ define $r=F^{-1}\left(  Y\right)  $, and if $Y<F\left(  0\right)
$ define $r=0$. Observe that for $r>0$ one has $Y=F\left(  r\right)  \geq
g\left(  r\vartheta\right)  \geq Cr$. Thus, in any case, $0\leq r\leq CY$.
Then
\[
\left\vert F\left(  \rho\right)  -Y\right\vert \geq F^{\prime}\left(
0\right)  \left\vert \rho-r\right\vert =\left(  g\left(  \vartheta\right)
-\nabla g\left(  k\right)  \cdot\vartheta\right)  \left\vert \rho-r\right\vert
=\Delta\left\vert \rho-r\right\vert .
\]

Hence,%
\begin{align*}
&  \int_{0}^{\delta Y}\rho^{\gamma-1}\left(  1+\left\vert g\left(
\rho\vartheta\right)  +g\left(  k-\rho\vartheta\right)  -Y\right\vert \right)
^{-\beta}d\rho\\
&  \leq C\int_{0}^{\delta Y}\rho^{\gamma-1}\left(  1+\Delta\left\vert
\rho-r\right\vert \right)  ^{-\beta}d\rho\\
&  =C\Delta^{-\gamma}\int_{0}^{\delta\Delta Y}t^{\gamma-1}\left(  1+\left\vert
t-\Delta r\right\vert \right)  ^{-\beta}dt.
\end{align*}

Call $\delta\Delta Y=T$ and $\Delta r=S$. Then we need to show that%
\[
\int_{0}^{T}t^{\gamma-1}\left(  1+\left\vert t-S\right\vert \right)  ^{-\beta
}dt \le CT^{\gamma-\tau}\log^{\sigma}\left(T\right).
\]

In any case we have%
\[
\int_{0}^{T}t^{\gamma-1}\left(  1+\left\vert t-S\right\vert \right)  ^{-\beta
}dt\leq\int_{0}^{T}t^{\gamma-1}dt\leq CT^{\gamma}.
\]
In particular, the thesis follows when $T\le 4$. 

If $T\geq4$ and $0\leq S\leq2$ then
\begin{align*}
&\int_{0}^{T}t^{\gamma-1}\left(  1+\left\vert t-S\right\vert \right)  ^{-\beta
}dt\\
&\leq C\left(  \int_{0}^{1}t^{\gamma-1}dt+\int_{1}^{T}t^{\gamma-\beta
-1}dt\right)\\
&  \leq T^{\gamma-\min\{\gamma,\beta\}}\log^{\delta_{\gamma,\beta}}(T).
\end{align*}

If $T\geq4$ and $2\leq S\leq T/2$ then%
\begin{align*}
&  \int_{0}^{T}t^{\gamma-1}\left(  1+\left\vert t-S\right\vert \right)
^{-\beta}dt\\
&  \leq\left(  \int_{0}^{S/2}+\int_{S/2}^{2S}+\int_{2S}^{T}\right)
t^{\gamma-1}\left(  1+\left\vert t-S\right\vert \right)  ^{-\beta}dt\\
&  \leq CS^{\gamma-\beta}+
S^{\gamma-\min\{1,\beta\}}\log^{\delta_{1,\beta}}(S)
  +CT^{\gamma-\min\{\gamma,\beta\}}\log^{\delta_{\gamma,\beta}}(T) \\
&
\leq
CT^{\gamma-\tau}\log^{\sigma}\left(T\right).
\end{align*}

If $T\geq4$ and $T/2\leq S\leq2T$, then%
\begin{align*}
&  \int_{0}^{T}t^{\gamma-1}\left(  1+\left\vert t-S\right\vert \right)
^{-\beta}dt\\
&  \leq CT^{-\beta}\int_{0}^{S/2}t^{\gamma-1}dt+CT^{\gamma-1}\int_{S/2}%
^{T}\left(  1+\left\vert t-S\right\vert \right)  ^{-\beta}dt\\
&  \leq CT^{\gamma-\beta}+
CT^{\gamma-\min\{1,\beta\}}\log^{\delta_{1,\beta}}(T) \\
&  \leq CT^{\gamma-\min\{1,\beta\}}\log^{\delta_{1,\beta}}(T).
\end{align*}

If $T\geq4$ and $S\geq2T$, then%
\[
\int_{0}^{T}t^{\gamma-1}\left(  1+\left\vert t-S\right\vert \right)  ^{-\beta
}dt\leq\int_{0}^{T}t^{\gamma-1}\left(  1+\left\vert t-T\right\vert \right)
^{-\beta}dt\leq
C
T^{\gamma-\min\{1,\beta\}}\log^{\delta_{1,\beta}}(T).
\]
\end{proof}

\begin{lemma}\label{lm_mu}
In the hypotheses of the previous lemma and with 
\[
\varsigma=\varsigma(d,\gamma,\beta)=\left\{
\begin{array}
[c]{ll}%
1 & \text{if }\tau=\left(  d-1\right)  /2\\
0 & \text{else,}%
\end{array}
\right.
\]
one has
\begin{align*}
&  \int_{\left\vert \vartheta\right\vert =1}\int_{0}^{\delta Y}\rho^{\gamma
-1}\left(  1+\left\vert g\left(  \rho\vartheta\right)  +g\left(
k-\rho\vartheta\right)  -Y\right\vert \right)  ^{-\beta}d\rho d\vartheta\\
&  \leq CY^{\gamma-\min\{  \tau,\left(  d-1\right)  /2\}  }%
\log^{\sigma+\varsigma}\left(  2+Y\right).
\end{align*}

\end{lemma}

\begin{proof}
Fix $k$. By the previous lemma and with the notation $\omega=k/\left\vert
k\right\vert $ and $\Delta\left(  \vartheta,\omega\right)  =g\left(
\vartheta\right)  -\nabla g\left(  \omega\right)  \cdot\vartheta$, we have to
estimate 
\begin{align*}
&  \int_{\left\vert \vartheta\right\vert =1}\int_{0}^{\delta Y}\rho^{\gamma
-1}\left(  1+\left\vert g\left(  \rho\vartheta\right)  +g\left(
k-\rho\vartheta\right)  -Y\right\vert \right)  ^{-\beta}d\rho d\vartheta\\
& \leq C Y^{\gamma}\int_{ \left\vert \vartheta\right\vert =1}\left(  1+Y\Delta\left(  \vartheta,\omega\right)  \right)  ^{-\tau}%
\log^{\sigma}\left(  2+Y\Delta\left(  \vartheta,\omega\right)  \right)
d\vartheta.
\end{align*}
By Lemma \ref{Support}, since $\Omega$ has everywhere positive curvature,
\[
c\left(  1-\vartheta\cdot\omega\right)  \leq\Delta\left(  \vartheta
,\omega\right)  \leq C\left(  1-\vartheta\cdot\omega\right)  .
\]%
Therefore,
\begin{align*}
&  Y^{\gamma}\int_{ \left\vert \vartheta\right\vert =1
}\left(  1+Y\Delta\left(  \vartheta,\omega\right)  \right)  ^{-\tau}%
\log^{\sigma}\left(  2+Y\Delta\left(  \vartheta,\omega\right)  \right)
d\vartheta\\
&  \leq CY^{\gamma}\int_{  \left\vert \vartheta\right\vert =1
}\left(  1+Y\left(  1-\vartheta\cdot\omega\right)  \right)  ^{-\tau}%
\log^{\sigma}\left(  2+Y\left(  1-\vartheta\cdot\omega\right)  \right)
d\vartheta\\
&  \leq CY^{\gamma}\int_{0}^{\pi}\left(  1+Y\left(  1-\cos\varphi\right)
\right)  ^{-\tau}\log^{\sigma}\left(  2+Y\left(  1-\cos\varphi\right)
\right)  \sin^{d-2}\varphi d\varphi\\
&  \leq CY^{\gamma}\int_{0}^{\pi}\left(  1+Y\varphi^{2}\right)  ^{-\tau}%
\log^{\sigma}\left(  2+Y\varphi^{2}\right)  \varphi^{d-2}d\varphi\\
&  \leq CY^{\gamma}\int_{0}^{Y^{-1/2}}\varphi^{d-2}d\varphi+CY^{\gamma-\tau
}\int_{Y^{-1/2}}^{\pi}\varphi^{d-2-2\tau}\log^{\sigma}\left(  2+Y\varphi
^{2}\right)  d\varphi\\
&  \leq CY^{\gamma-\left(  d-1\right)  /2}+
CY^{\gamma-\min\{\tau,(d-1)/2\}}\log^{\sigma+\varsigma}\left(  2+Y\right)\\
&  \leq CY^{\gamma-\min\{  \tau,\left(  d-1\right)  /2\} }%
\log^{\sigma+\varsigma}\left(  2+Y\right).
\end{align*}
\end{proof}

\begin{lemma}
\label{integral}
Let $g\left(  x\right)  $ be the
support function of $\Omega$, and let $d\geq2$, $d/2<\alpha<d$, $\beta\geq0$, $\zeta=\zeta(d,\alpha,\beta)=\min\left\{  1,\beta,d-\alpha,\left(
d-1\right)  /2\right\}$. Finally, define $\eta=\eta(d,\alpha,\beta)$ as follows.

If $d=2$ define
\[
\eta=\left\{
\begin{array}
[c]{ll}%
2 & \text{if }\beta=1/2\text{ and }\alpha=3/2\text{,}\\
1 & \text{if }\beta=1/2\text{ and }1<\alpha<3/2\text{,}\\
1 & \text{if }0<\beta<1/2\text{ and }\alpha=2-\beta\text{,}\\
1 & \text{if }\beta>1/2\text{ and }\alpha=3/2\text{,}\\
0 & \text{else.}%
\end{array}
\right.
\]

If $d=3$ define%
\[
\eta=\left\{
\begin{array}
[c]{ll}%
2 & \text{if }\beta=1\text{ and }3/2<\alpha\leq2\text{,}\\
1 & \text{if }\beta>1\text{ and }3/2<\alpha\leq2\text{,}\\
1 & \text{if }0<\beta<1\text{ and }\alpha=3-\beta\text{,}\\
0 & \text{else.}%
\end{array}
\right.
\]

If $d\geq4$ define%
\[
\eta=\left\{
\begin{array}
[c]{ll}%
1 & \text{if }\beta=1\text{ and }d/2<\alpha\leq d-1,\\
1 & \text{if }0<\beta<1\text{ and }\alpha=d-\beta,\\
0 & \text{else.}%
\end{array}
\right.
\]

Then there exists $C$\ such that for every
$k\in\mathbb{R}^{d}\setminus\left\{  0\right\}  $ and for every $-\infty<Y<+\infty$,
\begin{align*}
&
{\displaystyle\int_{\mathbb{R}^{d}}}
\left\vert x\right\vert ^{-\alpha}\left\vert k-x\right\vert ^{-\alpha}\left(
1+\left\vert g\left(  x\right)  +g\left(  k-x\right)  -Y\right\vert \right)
^{-\beta}dx\\
&  \leq C\left\vert k\right\vert ^{d-2\alpha}\left(  1+\left\vert k\right\vert
+\left\vert Y\right\vert \right)  ^{-\zeta  }\log^{\eta}\left(  2+\left\vert k\right\vert
+\left\vert Y\right\vert \right) .
\end{align*}
\end{lemma}

  \begin{figure}[h]
    \begin{subfigure}[t]{0.47\textwidth}
      \includegraphics[width=\textwidth]{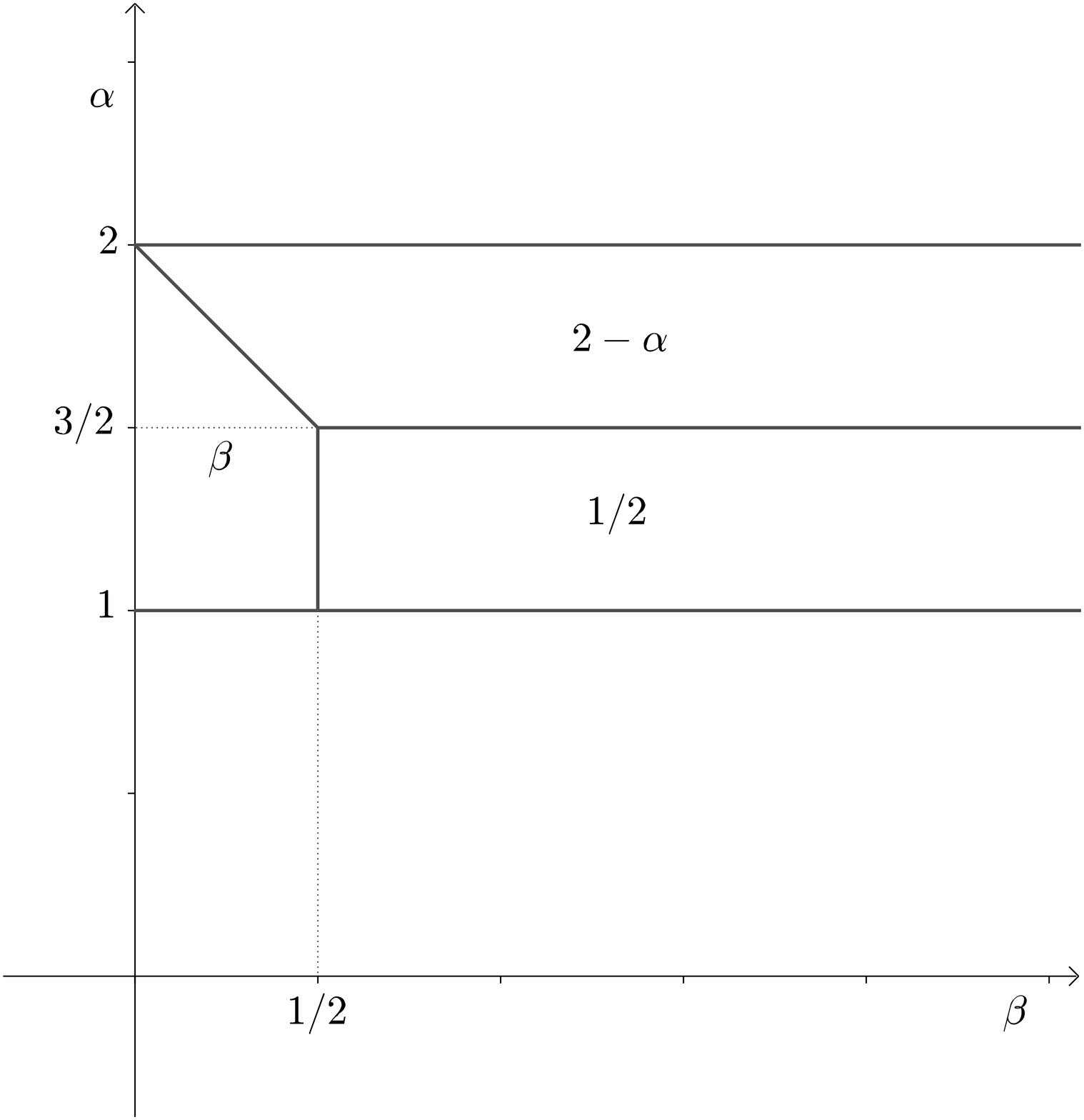}
      \caption{The case $d=2$.}
    \end{subfigure}\qquad
    \begin{subfigure}[t]{0.47\textwidth}
      \includegraphics[width=\textwidth]{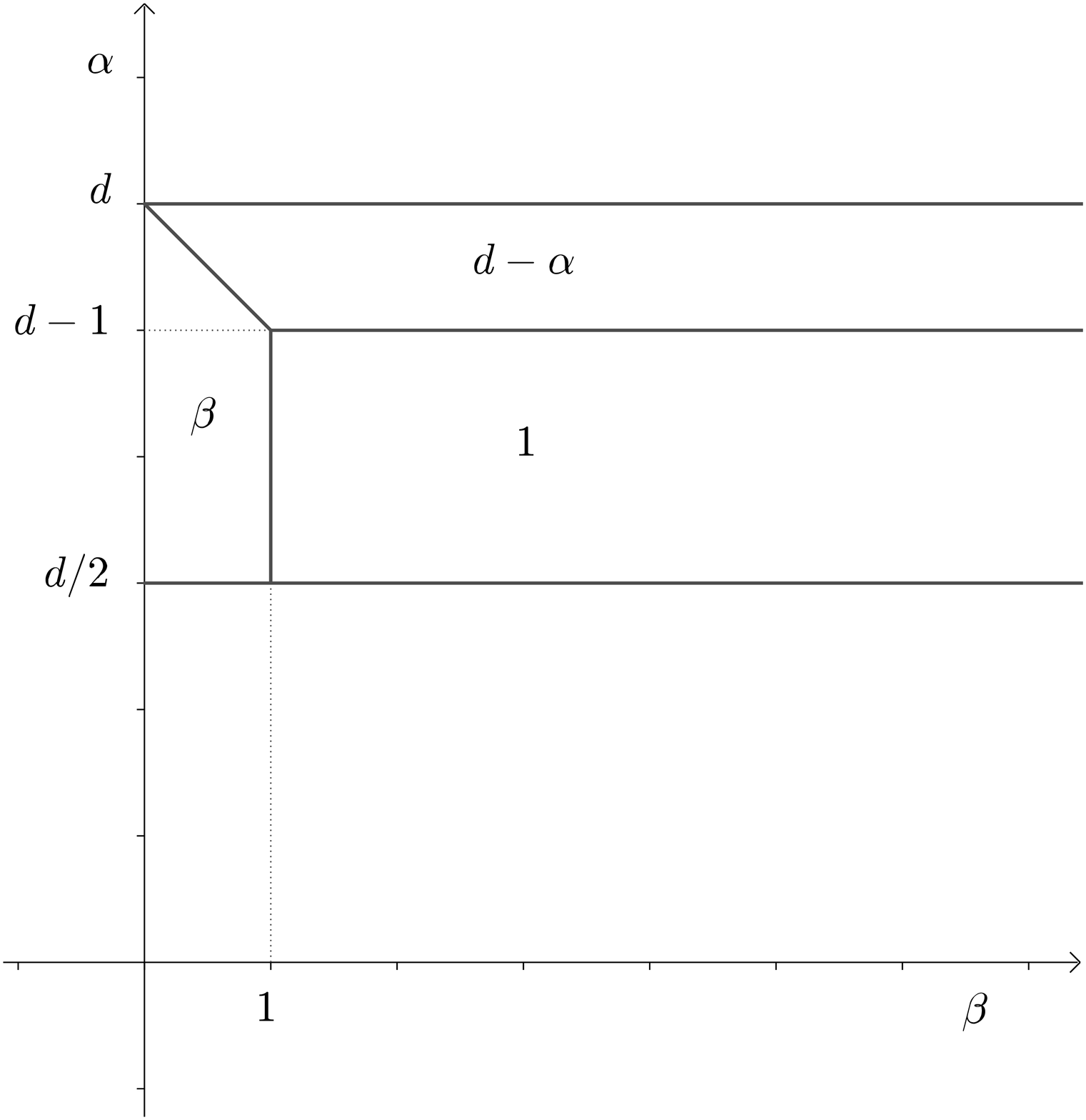}
      \caption{The case $d>2$.}
    \end{subfigure}
    \caption{The value of $\zeta$ as a function of $\beta$ and $\alpha$.}
  \end{figure}

\begin{proof}
Let us explain the numerology behind the lemma. If there is no cutoff $\left(
1+\left\vert g\left(  x\right)  +g\left(  k-x\right)  -Y\right\vert \right)
^{-\beta}$, then the change of variables $x=\left\vert k\right\vert z$ and
$k=\left\vert k\right\vert \omega$ gives
\[%
{\displaystyle\int_{\mathbb{R}^{d}}}
\left\vert x\right\vert ^{-\alpha}\left\vert k-x\right\vert ^{-\alpha
}dx=\left\vert k\right\vert ^{d-2\alpha}%
{\displaystyle\int_{\mathbb{R}^{d}}}
\left\vert z\right\vert ^{-\alpha}\left\vert \omega-z\right\vert ^{-\alpha
}dz=C\left\vert k\right\vert ^{d-2\alpha}.
\]
On the other hand, the cutoff $\left(  1+\left\vert g\left(  x\right)
+g\left(  k-x\right)  -Y\right\vert \right)  ^{-\beta}$ gives an extra decay.
In particular, the integral with the cutoff $\left(  1+\left\vert g\left(
x\right)  +g\left(  k-x\right)  -Y\right\vert \right)  ^{-\beta}$ with $\beta$
large is essentially over the $d-1$ dimensional set $\left\{  g\left(  x\right)  +g\left(
k-x\right)  =Y\right\}  $, that is the cutoff reduces the space dimension by
$1$. Hence, at least when $\beta$ is large, the integral with the
cutoff can be seen as the convolution in $\mathbb{R}^{d-1}$ of two homogeneous
functions of degree $-\alpha$, and this suggests the decay $\left\vert
k\right\vert ^{d-1-2\alpha}$. Hence, if with $\beta=0$ the decay is $\left\vert
k\right\vert ^{d-2\alpha}$, and if with $\beta>1$ the decay is $\left\vert
k\right\vert ^{d-1-2\alpha}$, then, by interpolation, when $0<\beta<1$ the
decay is $\left\vert k\right\vert ^{d-\beta-2\alpha}$. This is just a rough
numerology, indeed also the parameter $Y$ enters into play and the details of the proof
are more delicate.

For every $Y$ and $k$ one has
\begin{align*}
&
{\displaystyle\int_{\mathbb{R}^{d}}}
\left\vert x\right\vert ^{-\alpha}\left\vert k-x\right\vert ^{-\alpha}\left(
1+\left\vert g\left(  x\right)  +g\left(  k-x\right)  -Y\right\vert \right)
^{-\beta}dx\\
&  \leq%
{\displaystyle\int_{\mathbb{R}^{d}}}
\left\vert x\right\vert ^{-\alpha}\left\vert k-x\right\vert ^{-\alpha
}dx=C\left\vert k\right\vert ^{d-2\alpha}.
\end{align*}

Assume $\left\vert k\right\vert +\left\vert Y\right\vert \geq1.$ Since
$c\left\vert x\right\vert \leq g\left(  x\right)  \leq C\left\vert
x\right\vert $, one has
\[
c\left(  \left\vert k\right\vert +\left\vert x\right\vert \right)  \leq
g\left(  x\right)  +g\left(  k-x\right)  \leq C\left(  \left\vert k\right\vert
+\left\vert x\right\vert \right)  .
\]

Hence, if $-\infty<Y\leq\varepsilon\left\vert k\right\vert $ for a small
enough $\varepsilon>0$, one also has
\[
\left\vert g\left(  x\right)  +g\left(  k-x\right)  -Y\right\vert \geq
C\left(  \left\vert k\right\vert +\left\vert Y\right\vert \right)  .
\]

In this case $-\infty<Y\leq\varepsilon\left\vert k\right\vert $,
\begin{align*}
&
{\displaystyle\int_{\mathbb{R}^{d}}}
\left\vert x\right\vert ^{-\alpha}\left\vert k-x\right\vert ^{-\alpha}\left(
1+\left\vert g\left(  x\right)  +g\left(  k-x\right)  -Y\right\vert \right)
^{-\beta}dx\\
&  \leq C\left(  \left\vert k\right\vert +\left\vert Y\right\vert \right)
^{-\beta}%
{\displaystyle\int_{\mathbb{R}^{d}}}
\left\vert x\right\vert ^{-\alpha}\left\vert k-x\right\vert ^{-\alpha}dx\leq
C\left\vert k\right\vert ^{d-2\alpha}\left(  \left\vert k\right\vert
+\left\vert Y\right\vert \right)  ^{-\beta}.
\end{align*}

Assume now $\left\vert k\right\vert +\left\vert Y\right\vert \geq1$ and
$Y\geq\varepsilon\left\vert k\right\vert $. 
 Let us split the integral into the three sets
$\left\{  \left\vert x\right\vert +\left\vert k\right\vert \leq\varepsilon
Y\right\}  $, $\left\{  \varepsilon Y\leq\left\vert x\right\vert +\left\vert
k\right\vert \leq\delta Y\right\}  $ and $\left\{  \delta Y\leq\left\vert
x\right\vert +\left\vert k\right\vert <+\infty\right\}  $, with $\varepsilon$
small and $\delta$ large. In $\left\{  \left\vert x\right\vert +\left\vert
k\right\vert \leq\varepsilon Y\right\}  $ one has
\[
\left\vert g\left(  x\right)  +g\left(  k-x\right)  -Y\right\vert \geq CY.
\]

Hence%
\begin{align*}
&
{\displaystyle\int_{  \left\vert x\right\vert +\left\vert k\right\vert
\leq\varepsilon Y  }}
\left\vert x\right\vert ^{-\alpha}\left\vert k-x\right\vert ^{-\alpha}\left(
1+\left\vert g\left(  x\right)  +g\left(  k-x\right)  -Y\right\vert \right)
^{-\beta}dx\\
&  \leq CY^{-\beta}%
{\displaystyle\int_{\mathbb{R}^{d}}}
\left\vert x\right\vert ^{-\alpha}\left\vert k-x\right\vert ^{-\alpha}dx\leq
CY^{-\beta}\left\vert k\right\vert ^{d-2\alpha}\leq C\left\vert k\right\vert
^{d-2\alpha}\left(  \left\vert k\right\vert +\left\vert Y\right\vert \right)
^{-\beta}.
\end{align*}

In $\left\{  \delta Y\leq\left\vert x\right\vert +\left\vert k\right\vert
<+\infty\right\}  $ one has
\[
\left\vert g\left(  x\right)  +g\left(  k-x\right)  -Y\right\vert \geq
C\left(  \left\vert k\right\vert +\left\vert x\right\vert \right)  .
\]

Hence%
\begin{align*}
&
{\displaystyle\int_{ \delta Y\leq\left\vert x\right\vert +\left\vert
k\right\vert <+\infty  }}
\left\vert x\right\vert ^{-\alpha}\left\vert k-x\right\vert ^{-\alpha}\left(
1+\left\vert g\left(  x\right)  +g\left(  k-x\right)  -Y\right\vert \right)
^{-\beta}dx\\
&  \leq C%
{\displaystyle\int_{  \delta Y\leq\left\vert x\right\vert +\left\vert
k\right\vert <+\infty  }}
\left\vert x\right\vert ^{-\alpha}\left\vert k-x\right\vert ^{-\alpha}\left(
\left\vert k\right\vert +\left\vert x\right\vert \right)  ^{-\beta}dx\\
&  \leq CY^{-\beta}%
{\displaystyle\int_{\mathbb{R}^{d}}}
\left\vert x\right\vert ^{-\alpha}\left\vert k-x\right\vert ^{-\alpha}dx\\
&  \leq C\left\vert k\right\vert ^{d-2\alpha}\left(  \left\vert k\right\vert
+\left\vert Y\right\vert \right)  ^{-\beta}.
\end{align*}

It remains to estimate the integral over the spherical shell
\[
\left\{  \varepsilon Y\leq\left\vert x\right\vert +\left\vert k\right\vert
\leq\delta Y\right\}  \subseteq\left\{
\begin{array}
[c]{ll}%
\left\{  \left\vert x\right\vert \leq\delta Y\right\}  & \text{if }%
Y\leq4\left\vert k\right\vert /\varepsilon,\\
\left\{  \varepsilon Y/2\leq\left\vert x\right\vert \leq\delta Y\right\}  &
\text{if }Y\geq4\left\vert k\right\vert /\varepsilon.
\end{array}
\right.
\]

Recall that $Y\geq\varepsilon\left\vert k\right\vert $. Hence if
$\varepsilon\left\vert k\right\vert \leq Y\leq4\left\vert k\right\vert
/\varepsilon$, then%
\begin{align*}
&
{\displaystyle\int_{ \left\vert x\right\vert \leq\delta Y  }}
\left\vert x\right\vert ^{-\alpha}\left\vert k-x\right\vert ^{-\alpha}\left(
1+\left\vert g\left(  k+x\right)  +g\left(  k-x\right)  -Y\right\vert \right)
^{-\beta}dx\\
&  \leq%
{\displaystyle\int_{\left\{  \left\vert x-k/2\right\vert \leq\delta
Y+\left\vert k\right\vert /2\right\}  \cap\left\{  \left\vert x\right\vert
\leq\left\vert k-x\right\vert \right\}  }}
\left\vert x\right\vert ^{-\alpha}\left\vert k-x\right\vert ^{-\alpha}\left(
1+\left\vert g\left(  x\right)  +g\left(  k-x\right)  -Y\right\vert \right)
^{-\beta}dx\\
&  +%
{\displaystyle\int_{\left\{  \left\vert x-k/2\right\vert \leq\delta
Y+\left\vert k\right\vert /2\right\}  \cap\left\{  \left\vert x\right\vert
\geq\left\vert k-x\right\vert \right\}  }}
\left\vert x\right\vert ^{-\alpha}\left\vert k-x\right\vert ^{-\alpha}\left(
1+\left\vert g\left(  x\right)  +g\left(  k-x\right)  -Y\right\vert \right)
^{-\beta}dx\\
&  \leq C\left\vert k\right\vert ^{-\alpha}%
{\displaystyle\int_{  \left\vert x\right\vert \leq C\delta Y}}
\left\vert x\right\vert ^{-\alpha}\left(  1+\left\vert g\left(  x\right)
+g\left(  k-x\right)  -Y\right\vert \right)  ^{-\beta}dx.
\end{align*}

Similarly, if $Y\geq4\left\vert k\right\vert /\varepsilon$, then%
\begin{align*}
&
{\displaystyle\int_{  \varepsilon Y/2\leq\left\vert x\right\vert
\leq\delta Y  }}
\left\vert x\right\vert ^{-\alpha}\left\vert k-x\right\vert ^{-\alpha}\left(
1+\left\vert g\left(  x\right)  +g\left(  k-x\right)  -Y\right\vert \right)
^{-\beta}dx\\
&  \leq C%
{\displaystyle\int_{  \varepsilon Y/2\leq\left\vert x\right\vert
\leq\delta Y  }}
\left\vert x\right\vert ^{-2\alpha}\left(  1+\left\vert g\left(  x\right)
+g\left(  k-x\right)  -Y\right\vert \right)  ^{-\beta}dx.
\end{align*}

In polar coordinates $x=\rho\vartheta$ with $\rho\geq0$ and $\left\vert
\vartheta\right\vert =1$ the first integral takes the form%
\begin{align*}
&  \left\vert k\right\vert ^{-\alpha}%
{\displaystyle\int_{ \left\vert x\right\vert \leq C\delta Y}}
\left\vert x\right\vert ^{-\alpha}\left(  1+\left\vert g\left(  x\right)
+g\left(  k-x\right)  -Y\right\vert \right)  ^{-\beta}dx\\
&  =\left\vert k\right\vert ^{-\alpha}\int_{  \left\vert \vartheta
\right\vert =1  }\int_{0}^{C\delta Y}\rho^{d-\alpha-1}\left(
1+\left\vert g\left(  \rho\vartheta\right)  +g\left(  k-\rho\vartheta\right)
-Y\right\vert \right)  ^{-\beta}d\rho d\vartheta.
\end{align*}

Similarly, the second integral takes the form%
\begin{align*}
&
{\displaystyle\int_{  \varepsilon Y/2\leq\left\vert x\right\vert
\leq\delta Y  }}
\left\vert x\right\vert ^{-2\alpha}\left(  1+\left\vert g\left(  x\right)
+g\left(  k-x\right)  -Y\right\vert \right)  ^{-\beta}dx\\
&  =\int_{  \left\vert \vartheta\right\vert =1  }%
\int_{\varepsilon Y/2}^{\delta Y}\rho^{d-2\alpha-1}\left(  1+\left\vert
g\left(  \rho\vartheta\right)  +g\left(  k-\rho\vartheta\right)  -Y\right\vert
\right)  ^{-\beta}d\rho d\vartheta.
\end{align*}

By Lemma \ref{lm_mu}, the first integral can be bounded by%
\begin{align*}
&  \left\vert k\right\vert ^{-\alpha}\int_{  \left\vert \vartheta
\right\vert =1  }\int_{0}^{C\delta Y}\rho^{d-\alpha-1}\left(
1+\left\vert g\left(  \rho\vartheta\right)  +g\left(  k-\rho\vartheta\right)
-Y\right\vert \right)  ^{-\beta}d\rho d\vartheta\\
&  \leq C\left\vert k\right\vert ^{-\alpha}Y^{d-\alpha-\min\{
1,\beta,d-\alpha,\left(  d-1\right)  /2\}  }\log^{\sigma_{1}+\varsigma_{1}%
}\left(  2+Y\right)  \\
&  \leq C\left\vert k\right\vert ^{d-2\alpha}Y^{-\min\{  1,\beta
,d-\alpha,\left(  d-1\right)  /2\}  }\log^{\sigma_{1}+\varsigma_{1}}\left(
2+Y\right)  ,
\end{align*}
where
\begin{align*}
\varsigma_{1} &  =\left\{
\begin{array}
[c]{ll}%
1 & \text{if }\min\left\{  1,\beta,d-\alpha\right\}  =\left(  d-1\right)  /2\\
0 & \text{else.}%
\end{array}
\right.  \\
\sigma_{1} &  =\left\{
\begin{array}
[c]{ll}%
1 & \text{if }\beta=1\text{ and }d-\alpha\geq1,\\
1 & \text{if }\beta=d-\alpha\leq1,\\
0 & \text{else.}%
\end{array}
\right.
\end{align*}

Again by Lemma \ref{lm_mu}, the second integral can be bounded by%
\begin{align*}
&  \int_{  \left\vert \vartheta\right\vert =1  }\int%
_{\varepsilon Y/2}^{\delta Y}\rho^{d-2\alpha-1}\left(  1+\left\vert g\left(
\rho\vartheta\right)  +g\left(  k-\rho\vartheta\right)  -Y\right\vert \right)
^{-\beta}d\rho d\vartheta\\
&  \leq CY^{d-2\alpha-1}\int_{  \left\vert \vartheta\right\vert
=1  }\int_{0}^{\delta Y}\left(  1+\left\vert g\left(  \rho
\vartheta\right)  +g\left(  k-\rho\vartheta\right)  -Y\right\vert \right)
^{-\beta}d\rho d\vartheta\\
&  \leq CY^{d-2\alpha-\min\{  1,\beta,\left(  d-1\right)  /2\}  }%
\log^{\sigma_{2}+\varsigma_{2}}\left(  2+Y\right)  \\
&  \leq C\left\vert k\right\vert ^{d-2\alpha}Y^{-\min\{  1,\beta,\left(
d-1\right)  /2\}  }\log^{\sigma_{2}+\varsigma_{2}}\left(  2+Y\right)
\end{align*}
where%
\begin{align*}
\varsigma_{2} &  =\left\{
\begin{array}
[c]{ll}%
1 & \text{if }\min\left\{  1,\beta\right\}  =\left(  d-1\right)  /2\\
0 & \text{else.}%
\end{array}
\right.  \\
\sigma_{2} &  =\left\{
\begin{array}
[c]{ll}%
1 & \text{if }\beta=1,\\
0 & \text{else.}%
\end{array}
\right.
\end{align*}
It is easy to show that when $d\geq3$ then $\varsigma_{1}+\sigma_{1}\leq\eta$, and
that when $d\geq3$ and $\alpha\leq d-1$, then $\varsigma_{2}+\sigma_{2}\leq\eta$.
This shows the lemma in the case $d\geq3$. The case $d=2$ is more delicate.
First observe that the integral
\[%
{\displaystyle\int_{\mathbb{R}^{d}}}
\left\vert x\right\vert ^{-\alpha}\left\vert k-x\right\vert ^{-\alpha}\left(
1+\left\vert g\left(  x\right)  +g\left(  k-x\right)  -Y\right\vert \right)
^{-\beta}dx
\]
is a decreasing function of the variable $\beta$. It follows that for
$1/2<\beta<1$ and $\alpha=2-\beta$, one can take some $\tilde{\beta}\in\left(
1/2,\beta\right)  $ and obtain%
\begin{align*}
&
{\displaystyle\int_{\mathbb{R}^{2}}}
\left\vert x\right\vert ^{-\alpha}\left\vert k-x\right\vert ^{-\alpha}\left(
1+\left\vert g\left(  x\right)  +g\left(  k-x\right)  -Y\right\vert \right)
^{-\beta}dx\\
& \leq%
{\displaystyle\int_{\mathbb{R}^{2}}}
\left\vert x\right\vert ^{-\alpha}\left\vert k-x\right\vert ^{-\alpha}\left(
1+\left\vert g\left(  x\right)  +g\left(  k-x\right)  -Y\right\vert \right)
^{-\tilde{\beta}}dx\\
& \leq C\left\vert k\right\vert ^{2-2\alpha}Y^{-\min\{  1,\tilde{\beta
},2-\alpha,1/2\}  }\leq C\left\vert k\right\vert ^{2-2\alpha}%
Y^{-\min\{  1,\beta,2-\alpha,1/2\} }\\
& \leq C\left\vert k\right\vert ^{2-2\alpha}Y^{-1/2}.
\end{align*}
This shows that for $d=2$ and $1/2<\beta<1$ and $\alpha=2-\beta$ one can
indeed take $\eta=0.$ A similar argument shows that one can take $\eta=0$ 
also when $\beta=1$ and $1<\alpha<3/2$. 
In the remaining cases, it is easy to show that $\varsigma
_{1}+\sigma_{1}\leq\eta$, and that when $\alpha\leq3/2$, then $\varsigma_{2}%
+\sigma_{2}\leq\eta$.
\end{proof}

\begin{lemma}\label{lm_p=2}
Let $z_2=d/2$. If $\mathrm{Re}\left(  z\right)  \geq z_2$, there exists $C>0$ such that for every $R\ge1$ and $0<\delta<1/2$,
\[
\int_{\mathbb{R}}  \int_{\mathbb{T}^{d}}|\Phi\left(  \delta,z,r,x\right)  | ^{2}dxd\mu(r-R)\leq
\begin{cases}
C & \text{if }\mathrm{Re}\left(  z\right)  >z_2,\\
C\log\left(  1/\delta\right)  & \text{if }\mathrm{Re}\left(  z\right)=z_2.
\end{cases}
\]

\end{lemma}

\begin{proof}
By Plancherel formula applied to $\Phi\left(  \delta,z,r,x\right)  $
as a function of the variable $x$,
\begin{align*}
&  \int_{\mathbb{R}}\int_{\mathbb{T}^{d}}| \Phi\left(\delta,z,r,x\right)  | ^{2}dxd\mu(r-R)\\
&  \leq C\sum_{j=0}^h\sum_{m\in\mathbb{Z}^{d},\,m\neq0}| \widehat{\varphi}\left(  \delta m\right)  |^{2}| m| ^{-2\mathrm{Re}\left(  z\right)-2j  }\int_{\mathbb{R}} r^{-2j}d\mu(r-R).
\end{align*}
Observe that if $R>1$ then the singularity of $r^{-j}$ is outside the support of $d\mu(r-R)$, and recall that 
$\widehat\varphi(r)$ has fast decay at infinity. Hence for every $\lambda>0$,
\[
\int_{\mathbb{R}}\int_{\mathbb{T}^{d}}| \Phi\left(\delta,z,r,x\right)  | ^{2}dxd\mu(r-R) \leq C\sum_{m\in\mathbb{Z}^{d},\,m\neq0}\left(  1+\delta|m| \right)  ^{-\lambda}| m| ^{-2\mathrm{Re}\left(  z\right)  }
\]
and the lemma follows.
\end{proof}

The following lemma is an estimate of the $L^p  $ norms of the function $\Phi\left(  \delta,z,r,x\right)  $ when $p=4$ and the space dimension $d\geq2$. 
In dimension $d=2$ there is a second relevant exponent $p=6$, and this will be considered later.

  \begin{figure}[h]
    \begin{subfigure}[t]{0.47\textwidth}
      \includegraphics[width=\textwidth]{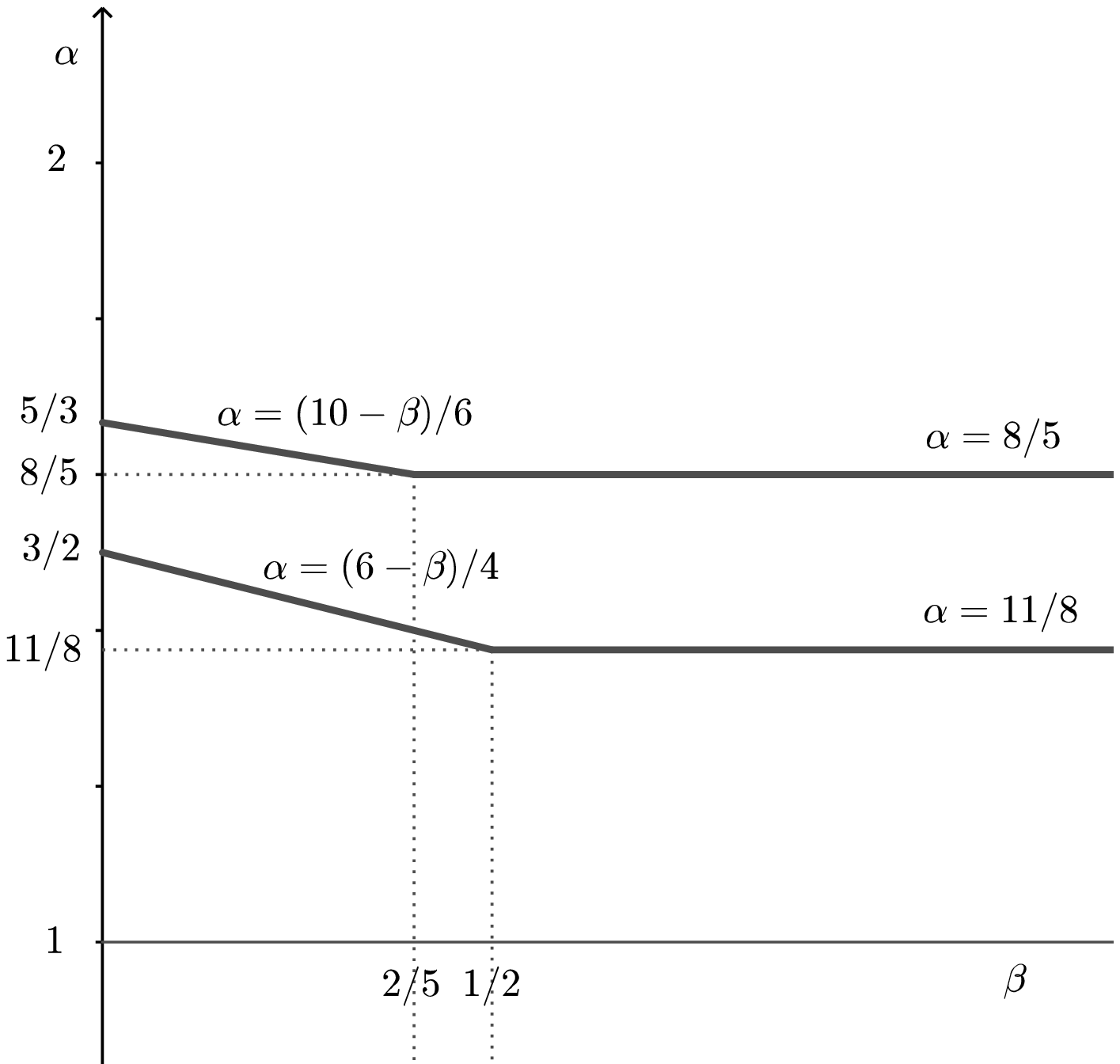}
      \caption{The case $d=2$ with $z_4$ (bottom) and $z_6$ (top).}
    \end{subfigure}\qquad
    \begin{subfigure}[t]{0.47	\textwidth}
      \includegraphics[width=\textwidth]{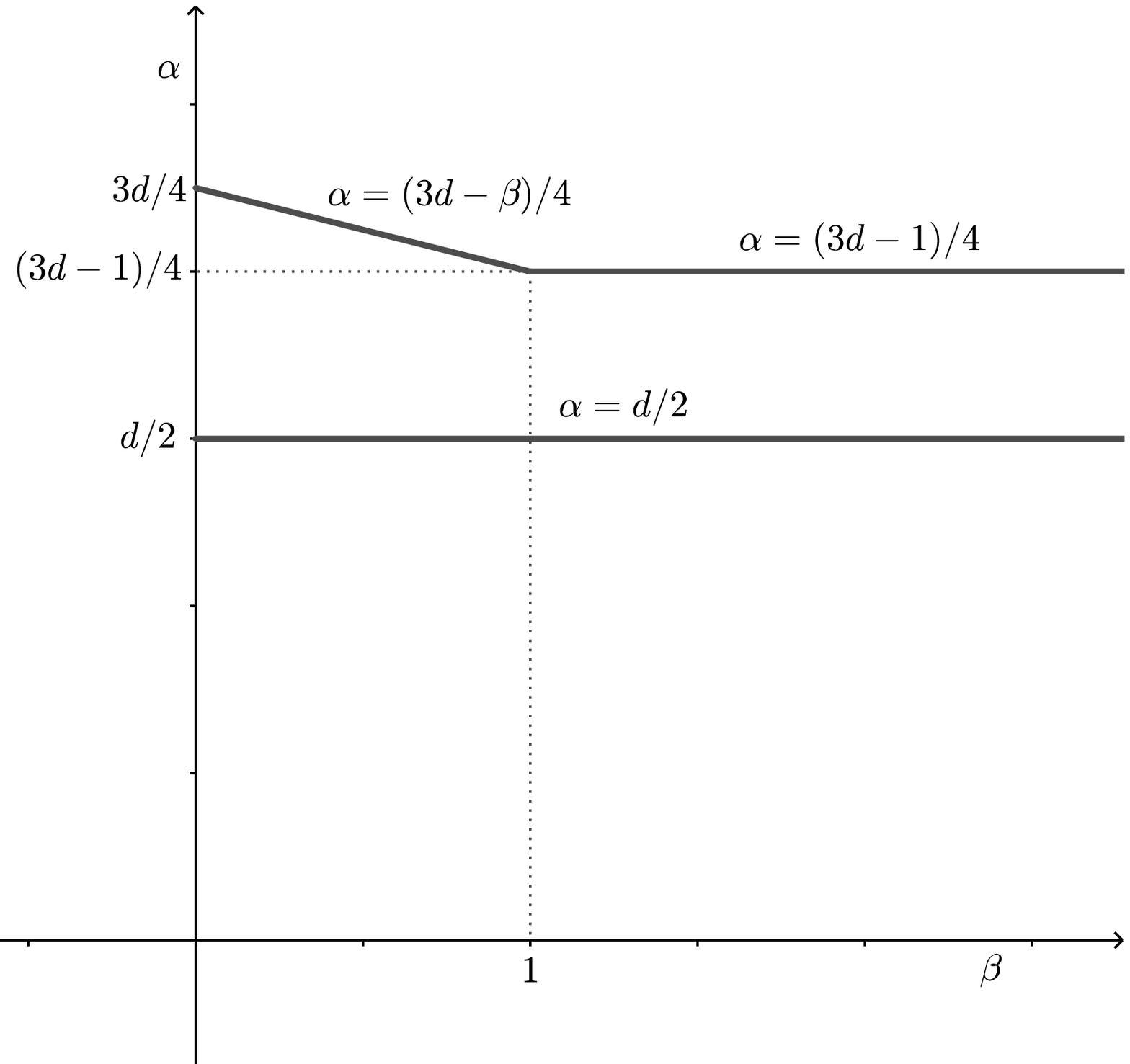}
      \caption{The case $d>2$ with $z_2$ (bottom) and $z_4$ (top).}
    \end{subfigure}
    \caption{The minimal values of $\alpha=\mathrm{Re}(z)$ as a function of $\beta$.}\label{F3}
  \end{figure}

\begin{lemma}
\label{lm_p=4}
Let $z\in\mathbb C$. 
Define $\nu=\min \{1,(d-1)/2\}$ and let $\eta=\eta(d,\mathrm{Re}(z),\beta)$ be defined as in Lemma \ref{integral}.
Let $z_4=\max\{(3d-\beta)/4,\,(3d-\nu)/4\}$. If $\mathrm{Re}(z)  \geq z_4$, then there exists $C>0$ such that for every $R\geq 1$ and $0<\delta<1/2$,
\begin{align*}
  \int_{\mathbb{R}}  \int_{\mathbb{T}^{d}}|\Phi\left(  \delta,z,r,x\right)  | ^{4}dxd\mu(r-R)
  \leq
\begin{cases}
C & \text{if }\mathrm{Re}\left(  z\right)  >z_4,\\
C\log^{\eta+1}\left(  1/\delta\right)  & \text{if }\mathrm{Re}\left(  z\right)=z_4.\\
\end{cases}
\end{align*}
\end{lemma}

\begin{proof}
Call $\alpha=\mathrm{Re}\left(  z\right)$.
By the above Lemma \ref{lm_6_integral_ellipsoid} with $N=2$, it suffices to estimate 
\begin{align*}
&  \underset{\mathbb{R}^{d}}\int\left(  1+\delta| k| \right)^{-\lambda}\underset{| m| ,| k-m|>1}\int| m| ^{-\alpha}| k-m| ^{-\alpha}\\
&  \times\underset{{| n| ,| k-n|>1}}\int| n| ^{-\alpha}| k-n| ^{-\alpha}\left(  1+| g( m) +g( k-m)-g(n) -g( k-n) | \right)^{-\beta}dn dm dk.
\end{align*}
Notice that we have canceled all the cutoff functions in the variables $m$, $k-m$, $n$, $k-n$. 
The integral over the set $\left\{  | k| \leq2\right\}  $ is bounded by
\begin{align*}
&  \int_{  | k| \leq2  }\int_{\substack{| m| ,| k-m| >1}}|m| ^{-\alpha}| k-m| ^{-\alpha}\int_{\substack{| n| ,| k-n| >1}}|n| ^{-\alpha}| k-n| ^{-\alpha}dn dm dk\\
& \leq \int_{| k| \leq2} {dk}  \int_{\substack{|m| >1}}| m| ^{-2\alpha}dm  \int_{\substack{|n| >1}}| n| ^{-2\alpha}dn
\leq C.
\end{align*}
Let us now consider the integral over the set $\left\{  |k| \geq2\right\}  $,
\begin{align*}
&  \underset{|k|\geq 2}\int\left(  1+\delta| k| \right)^{-\lambda}\underset{{| m| ,| k-m|>1}}\int| m| ^{-\alpha}| k-m| ^{-\alpha}\\
&  \times\underset{{| n| ,| k-n|>1}}\int| n| ^{-\alpha}| k-n| ^{-\alpha}\left(  1+| g( m) +g( k-m)-g(n) -g( k-n) | \right)^{-\beta}dn dm dk.
\end{align*}
By Lemma \ref{integral}, the inner integral
\[
\int_{\substack{| n| ,| k-n|>1}}| n| ^{-\alpha}| k-n| ^{-\alpha}\left(  1+| g( m) +g( k-m)-g(n) -g( k-n) | \right)^{-\beta}dn
\]
is bounded by
\begin{align*}
&C| k| ^{d-2\alpha}(1+|k|+g(m)+g(k-m))^{-\zeta}\log^{\eta}(2+|k|+g(m)+g(k-m))\\
&\le C| k| ^{d-2\alpha}(1+|k|+|m|)^{-\zeta}\log^{\eta}(2+|k|+|m|).
\end{align*}
Thus, the goal estimate becomes
\begin{align*}
&  \underset{{  | k| >2}  }\int\left(1+\delta| k| \right)  ^{-\lambda}| k|^{d-2\alpha}\underset{\mathbb{R}^{d}}\int| m| ^{-\alpha}| k-m| ^{-\alpha}\left(  1+| k|+| m| \right)  ^{-\zeta}\log^{\eta}(|k|+|m|)dmdk\\
&  \leq C\int_{\substack{  | k| >2}  }\left(1+\delta| k| \right)  ^{-\lambda}| k|^{d-2\alpha-\zeta}
\log^{\eta}(|k|)
\int_{|m|\leq 2| k| }| m| ^{-\alpha}| k-m| ^{-\alpha}dmdk\\
&  + C\int_{  | k| >2  }\left(1+\delta| k| \right)  ^{-\lambda}| k|
^{d-2\alpha}\int_{  |m|\geq 2| k|  }| m|^{-2\alpha-\zeta} \log^{\eta}(|m|)dmdk\\
&\leq \int_{  | k| >2  }\left(1+\delta| k| \right)  ^{-\lambda}| k|
^{2d-4\alpha-\zeta} \log^{\eta}(|k|)dk\\
&\leq
\begin{cases}
C & \text{if }\alpha>\left(  3d-\zeta\right)  /4,\\
C\log^{\eta+1}\left(  2/\delta\right)  & \text{if }\alpha=\left(  3d-\zeta\right)/4.
\end{cases}
\end{align*}
The result now follows after the observation that $\alpha>(3d-\zeta)/4$ if and only if $\alpha>z_4$
 and $\alpha=(3d-\zeta)/4$ if and only if $\alpha=z_4$.
\end{proof}

In the following lemma the space dimension is $d=2$.
 
 \begin{lemma} \label{lm_p=6} 
 Let $d=2$ and let $z_6=\max\{(10-\beta)/6,8/5\}$.
If $\mathrm{Re}(z)  \geq z_6$, then there exists $C>0$ such that for every $R\geq 1$ and $0<\delta<1/2$,
\begin{align*}
  \int_{\mathbb{R}}  \int_{\mathbb{T}^{2}}|\Phi\left(  \delta,z,r,x\right)  | ^{6}dxd\mu(r-R)
  \leq
\begin{cases}
C & \text{if }\mathrm{Re}\left(  z\right)  >z_6,\\
C\log\left(  1/\delta\right)  & \text{if }\mathrm{Re}\left(  z\right)=z_6 \text { and }
\beta\neq 2/5,\\
C\log^2\left(  1/\delta\right)  & \text{if }\mathrm{Re}\left(  z\right)=z_6 \text { and }
\beta= 2/5.\\
\end{cases}
\end{align*}
 \end{lemma}
 
 \proof
Call $\alpha=\mathrm{Re}\left(  z\right)$. By Lemma \ref{lm_6_integral_ellipsoid} with $N=3$, it suffices to estimate
 \begin{align*}
 &  \int_{\mathbb{R}^{2}} \left(  1+\delta| k| \right)^{-\lambda} \underset{|m_1| ,|m_2|,| k-m_1-m_2| >1}{\int\int}|m_1|^{-\alpha}|m_2|^{-\alpha}|k-m_1-m_2|^{-\alpha}\\
 & \times \underset{|n_1| ,|n_2|,| k-n_1-n_2| >1}{\int\int}|n_1|^{-\alpha}|n_2|^{-\alpha}|k-n_1-n_2|^{-\alpha}\\
 &  \times\left(  1+| g(m_1) +g(m_2)+g(k-m_1-m_2) -g(n_1)-g(n_2)-g(k-n_1-n_2) |\right)  ^{-\beta}\\
 & \times dn_1 dn_2 dm_1 dm_2 dk.
 \end{align*}
Split $\mathbb{R}^{2}$ as $\left\{|k| \leq2\right\}  \cup\left\{  | k|
 \geq2\right\}  $. The integral over the disc $\left\{ | k| \leq2\right\}  $ is bounded by
 \begin{align*}
 &  \int_{ | k| \leq2  }\left(\int_{|m_1|>1}| m_1| ^{-\alpha}\int_{\mathbb{R}^{2}}| m_2| ^{-\alpha}| k-m_1-m_2| ^{-\alpha}dm_2dm_1\right)  ^{2}dk\\
 &  =C\int_{ | k| \leq2  }\left(\int_{|m_1|>1}| m_1| ^{-\alpha}| k-m_1| ^{2-2\alpha}dm_1\right)  ^{2}dk\\
 &  =C\int_{  | k| \leq1/2 }\left(\int_{|m_1|>1}| m_1| ^{2-3\alpha}dm_1\right)  ^{2}dk
 +C\int_{1/2\leq | k| \leq2  }|k| ^{8-6\alpha}dk\leq C,
 \end{align*}
 since $\alpha\ge z_6>4/3$.
 Consider now the case $\left\{  | k| \geq2\right\}  $. Apply Lemma \ref{integral} to the integral with respect to $n_2$ with $k$ replaced with $k-n_1$ and
 $Y$ replaced with $g(m_1)+g(m_2)+g(k-m_1-m_2)-g(n_1)$,
  \begin{align*}
 & \underset{|n_1| ,|n_2|,| k-n_1-n_2| >1}{\int\int}|n_1|^{-\alpha}|n_2|^{-\alpha}|k-n_1-n_2|^{-\alpha}\\
&\times \left(  1+| g(m_1) +g(m_2)+g(k-m_1-m_2) -g(n_1)-g(n_2)-g(k-n_1-n_2) |\right)  ^{-\beta}\\
&\times  dn_2dn_1 \\
 &  \leq C\int_{\mathbb R^2}|n_1|^{-\alpha}|k-n_1|^{2-2\alpha} \\
& \times\left(  1+ |k-n_1|+ |g(m_1) +g(m_2)+g(k-m_1-m_2) -g(n_1) |\right)  ^{-\zeta}\\
& \times \log^\eta(2+ |k-n_1|+ |g(m_1) +g(m_2)+g(k-m_1-m_2) -g(n_1))dn_2dn_1\\
   &  \leq \int_{\mathbb R^2}|n_1|^{-\alpha}|k-n_1|^{2-2\alpha} 
   \left(1+|k-n_1| \right)^{-\zeta}
\log^\eta(2+|k-n_1|)
  dn_1\\
 &  \leq C|k|^{4-3\alpha-\zeta}\log^\eta(|k|).
 \end{align*}
 Here $\eta=\eta(2,\alpha,\beta)$ as defined in Lemma \ref{integral}.
 Moreover,
 \begin{align*}
 &\int_{\mathbb{R}^{2}}| m_1| ^{-\alpha}\int_{\mathbb{R}^{2}}| m_2| ^{-\alpha}| k-m_1-m_2| ^{-\alpha}dm_2dm_1\\
 &=C\int_{\mathbb{R}^{2}}| m_1| ^{-\alpha}|k-m_1| ^{2-2\alpha}dm_1=C| k| ^{4-3\alpha}.
 \end{align*}
 Finally, the integral over $\left\{  | k| \geq2\right\}  $ gives
 $$\int_{\substack{  | k| \geq2}  }(1+\delta|k|)^{-\lambda}|k| ^{8-6\alpha-\zeta}\log^\eta(|k|)dk\leq 
 \begin{cases}
 C & \text{ if }\alpha>(10-\zeta)/6,\\
 C\log^{\eta+1}(1/\delta) &  \text{ if }\alpha=(10-\zeta)/6.
 \end{cases}$$
 The result now follows after the observation that $\alpha>(10-\zeta)/6$ if and only if $\alpha>z_6$
 and $\alpha=(10-\zeta)/6$ if and only if $\alpha=z_6$.
  \endproof

 \begin{lemma} \label{lm_Interpolation} 
The notation is as in the previous lemmas.
 \begin{itemize}
 \item[(1)] Let $d=3$. If $\beta<1$ then there exists a constant $C$ such that for every  $\mathrm{Re}\left(  z\right)  \geq 2$, for every $R\ge 1$ and $0<\delta<1/2$,
 \begin{align*}
   \left\{  \int_{\mathbb{R}}  \int_{\mathbb{T}^{d}}| \Phi\left(  \delta,z,r,x\right)  | ^{p}dxd\mu(r-R)\right\}^{1/p}
   \leq
 \begin{cases}
 C & \text{if }p<3+\beta/(2-\beta),\\
  C\log^{1/p}\left(  1/\delta\right)  & \text{if }p=3+\beta/(2-\beta).
 \end{cases}
 \end{align*}
 If $\beta=1$ then there exists a constant $C$ such that for every  $\mathrm{Re}\left(  z\right)  \geq 2$,  for every $R\ge 1$ and $0<\delta<1/2$,
 \begin{align*}
   \left\{  \int_{\mathbb{R}}  \int_{\mathbb{T}^{d}}| \Phi\left(  \delta,z,r,x\right)  | ^{p}dxd\mu(r-R)\right\}^{1/p}
   \leq
 \begin{cases}
 C & \text{if }p<4,\\
  C\log^{3/4}\left(  1/\delta\right)  & \text{if }p=4.
 \end{cases}
 \end{align*}
 If $\beta>1$ then there exists a constant $C$ such that for every  $\mathrm{Re}\left(  z\right)  \geq 2$,  for every $R\ge 1$ and $0<\delta<1/2$,
 \begin{align*}
  \left\{  \int_{\mathbb{R}}  \int_{\mathbb{T}^{d}}| \Phi\left(  \delta,z,r,x\right)  | ^{p}dxd\mu(r-R)\right\}^{1/p}
  \leq
 \begin{cases}
 C & \text{if }p<4,\\
  C\log^{1/2}\left(  1/\delta\right)  & \text{if }p=4.
 \end{cases}
 \end{align*}

 \item[(2)] Let $d\ge4$. If $\beta<1$ then there exists a constant $C$ such that for every  $\mathrm{Re}\left(  z\right)  \geq (d+1)/2$,  for every $R\ge 1$ and $0<\delta<1/2$,
 \begin{align*}
  & \left\{  \int_{\mathbb{R}}  \int_{\mathbb{T}^{d}}| \Phi\left(  \delta,z,r,x\right)  | ^{p}dxd\mu(r-R)\right\}^{1/p}\\
  & \leq
 \begin{cases}
 C & \text{if }p<2(d-\beta)/(d-\beta-1),\\
  C\log^{1/p}\left(  1/\delta\right)  & \text{if }p=2(d-\beta)/(d-\beta-1).
 \end{cases}
 \end{align*}
If $\beta=1$ then there exists a constant $C$ such that for every  $\mathrm{Re}\left(  z\right)  \geq (d+1)/2$,  for every $R\ge 1$ and $0<\delta<1/2$,
 \begin{align*}
   &\left\{  \int_{\mathbb{R}}  \int_{\mathbb{T}^{d}}| \Phi\left(  \delta,z,r,x\right)  | ^{p}dxd\mu(r-R)\right\}^{1/p}\\
   &\leq
 \begin{cases}
 C & \text{if }p<2(d-1)/(d-2),\\
  C\log^{1/2}\left(  1/\delta\right)  & \text{if }p=2(d-1)/(d-2).
 \end{cases}
 \end{align*}
 If $\beta>1$ then there exists a constant $C$ such that for every  $\mathrm{Re}\left(  z\right)  \geq (d+1)/2$,  for every $R\ge 1$ and $0<\delta<1/2$,
 \begin{align*}
  &\left\{  \int_{\mathbb{R}}  \int_{\mathbb{T}^{d}}| \Phi\left(  \delta,z,r,x\right)  | ^{p}dxd\mu(r-R)\right\}^{1/p}\\
  &\leq
 \begin{cases}
 C & \text{if }p<2(d-1)/(d-2),\\
  C\log^{1/p}\left(  1/\delta\right)  & \text{if }p=2(d-1)/(d-2).
 \end{cases}
 \end{align*}
 \end{itemize}
 \end{lemma}
 
 \begin{proof}
 It is enough to prove the result for $z=(d+1)/2$.
 The Lemma follows 
 via complex interpolation. For the definition of the complex interpolation method and the complex interpolation of $L^p  $ spaces, see for example \cite[Chapters 4 and 5]{BL}. Here we recall the relevant result: Let $\mathbb{X}$ be a measure space, $1\leq a<b\leq+\infty$, $-\infty<A<B<+\infty$, and let $\Phi\left(  z\right)  $ be a function with values in $L^{a}\left(  \mathbb{X}\right)  +L^{b}\left(  \mathbb{X}\right)  $, continuous and bounded on the closed strip $\left\{  A\leq\mathrm{Re}\left(  z\right)  \leq B\right\}  $ and analytic on the open strip $\left\{A<\mathrm{Re}\left(  z\right)  <B\right\}  $. Assume that there exist constants $H$ and $K$ such that for every $-\infty<t<+\infty$,
 $$
 \begin{cases}
 \left\Vert \Phi\left(  A+it\right)  \right\Vert _{L^{a}\left(  \mathbb{X}\right)  }\leq H,\\
 \left\Vert \Phi\left(  B+it\right)  \right\Vert _{L^{b}\left(  \mathbb{X}\right)  }\leq K.
 \end{cases}
 $$
 If $1/p=\left(  1-\vartheta\right)  /a+\vartheta/b$, with $0<\vartheta<1$, then
 $$\left\Vert \Phi\left(  \left(  1-\vartheta\right)  A+\vartheta B\right)\right\Vert _{L^{p}\left(  \mathbb{X}\right)  }\leq H^{1-\vartheta}K^{\vartheta}.$$
In our case, the analytic function is $\Phi\left(  \delta,z,r,x\right)  $, the measure space is $\mathbb{R\times T}^{d}$ with measure $d\mu(r-R)dx$, $a=2$, $b=4$, $A=z_2+\varepsilon$, $B=z_4+\varepsilon$, with $\varepsilon\geq0$. The norms $H$ and $K$ are given in Lemma \ref{lm_p=2} and Lemma \ref{lm_p=4}. 
 Set
 $$\frac{d+1}{2}=\left(  1-\vartheta\right)  A+\vartheta B.$$
 This gives
 $$\vartheta=\frac{d+1-2A}{2B-2A},$$
 and
 $$p=\frac{2(d-\min\{\beta,1\}) }{d-1-\min\{\beta,1\}+2\varepsilon}.$$
 When $\varepsilon>0$ and $p<{2(d-\min\{\beta,1\}) }/{(d-1-\min\{\beta,1\})}  $,
 $$\left\{  \int_{\mathbb{R}} \int_{\mathbb{T}^{d}}| \Phi\left(  \delta,\left(  d+1\right)  /2,r,x\right)  |^{p}dxd\mu(r-R)\right\}  ^{1/p}\leq C.$$
 When $\varepsilon=0$ and $p={2(d-\min\{\beta,1\}) }/{(d-1-\min\{\beta,1\})}   $,
 $$\left\{  \int_{\mathbb{R}}  \int_{\mathbb{T}^{d}}| \Phi\left(  \delta,\left(  d+1\right)  /2,r,x\right)  |^{p}dxd\mu(r-R)\right\}  ^{1/p}\leq C\log^{1/p+\eta/(2d-2\min\{\beta,1\})}\left(  1/\delta\right)  ,$$
 where $\eta=\eta(d,z_4,\beta)$ as in Lemma \ref{integral}.
 \end{proof}

  \begin{lemma} \label{lm_Interpolation_d=2} 
The notation is as in the previous lemmas and let $d=2$. 

\noindent If $0\leq\beta<2/5$ then there exist a constant $C$ such that for every  $\mathrm{Re}\left(  z\right)  \geq 3/2$,  for every $R\ge 1$ and $0<\delta<1/2$,
 \begin{align*}
   \left\{  \int_{\mathbb{R}}  \int_{\mathbb{T}^{2}}| \Phi\left(  \delta,z,r,x\right)  | ^{p}dxd\mu(r-R)\right\}^{1/p}
   \leq
 \begin{cases}
 C & \text{if }p<4+2\beta,\\
  C\log^{1/p}\left(  1/\delta\right)  & \text{if }p=4+2\beta.
 \end{cases}
 \end{align*}
 If $\beta=2/5$ then there exist a constant $C$ such that for every  $\mathrm{Re}\left(  z\right)  \geq 3/2$,  for every $R\ge 1$ and $0<\delta<1/2$,
 \begin{align*}
   \left\{  \int_{\mathbb{R}}  \int_{\mathbb{T}^{2}}| \Phi\left(  \delta,z,r,x\right)  | ^{p}dxd\mu(r-R)\right\}^{1/p}
   \leq
 \begin{cases}
 C & \text{if }p<4+2\beta,\\
  C\log^{1/p+1/12}\left(  1/\delta\right)  & \text{if }p=4+2\beta.
 \end{cases}
 \end{align*}
 If $2/5<\beta<1/2$ then there exist a constant $C$ such that for every  $\mathrm{Re}\left(  z\right)  \geq 3/2$,  for every $R\ge 1$ and $0<\delta<1/2$,
 \begin{align*}
  \left\{  \int_{\mathbb{R}}  \int_{\mathbb{T}^{2}}| \Phi\left(  \delta,z,r,x\right)  | ^{p}dxd\mu(r-R)\right\}^{1/p}
  \leq
 \begin{cases}
 C & \text{if }p<4+10\beta/(3+5\beta),\\
  C\log^{1/p}\left(  1/\delta\right)  & \text{if }p=4+10\beta/(3+5\beta).
 \end{cases}
 \end{align*}
  If $\beta=1/2$ then there exist a constant $C$ such that for every  $\mathrm{Re}\left(  z\right)  \geq 3/2$,  for every $R\ge 1$ and $0<\delta<1/2$,
 \begin{align*}
  \left\{  \int_{\mathbb{R}}  \int_{\mathbb{T}^{2}}| \Phi\left(  \delta,z,r,x\right)  | ^{p}dxd\mu(r-R)\right\}^{1/p}
  \leq
 \begin{cases}
 C & \text{if }p<4+10/11,\\
  C\log^{1/p+1/9}\left(  1/\delta\right)  & \text{if }p=4+10/11.
 \end{cases}
 \end{align*}
  If $\beta>1/2$ then there exist a constant $C$ such that for every  $\mathrm{Re}\left(  z\right)  \geq 3/2$,  for every $R\ge 1$ and $0<\delta<1/2$,
 \begin{align*}
  \left\{  \int_{\mathbb{R}}  \int_{\mathbb{T}^{2}}| \Phi\left(  \delta,z,r,x\right)  | ^{p}dxd\mu(r-R)\right\}^{1/p}
  \leq
 \begin{cases}
 C & \text{if }p<4+10/11,\\
  C\log^{1/p}\left(  1/\delta\right)  & \text{if }p=4+10/11.
 \end{cases}
 \end{align*}

 \end{lemma}
 
 \begin{proof}
 Again, it is enough to prove the result for $z=3/2$. The case $\beta=0$ is contained in Lemma \ref{lm_p=4}. If $\beta>0$
the proof follows by complex interpolation with $a=4$, $b=6$, $A=z_4+\varepsilon$, $B=z_6+\varepsilon$, with $\varepsilon\geq0$. The norms $H$ and $K$ are given in Lemma \ref{lm_p=4} and Lemma \ref{lm_p=6}. 
 Set
 $$\frac{3}{2}=\left(  1-\vartheta\right)  A+\vartheta B.$$
 This gives
 $$\vartheta=\frac{3-2A}{2B-2A},$$
 and
 $${p}=\frac{24(z_6-z_4)}{6z_6-4z_4-3+2\varepsilon}.$$
 When $\varepsilon>0$ and $p<{24(z_6-z_4)}/{(6z_6-4z_4-3)} $,
 $$\left\{  \int_{\mathbb{R}} \int_{\mathbb{T}^{2}}| \Phi\left(  \delta,3/2,x\right)  |^{p}dxd\mu(r-R)\right\}  ^{1/p}\leq C.$$
 When $\varepsilon=0$ and $p={24(z_6-z_4)}/{(6z_6-4z_4-3)} $,
 $$\left\{  \int_{\mathbb{R}}  \int_{\mathbb{T}^{d}}| \Phi\left(  \delta,3  /2,r,x\right)  |^{p}dxd\mu(r-R)\right\}  ^{1/p}\leq C\log^{1/p+\eta(1-\vartheta)/4+(\omega-1)\vartheta/6}\left(  1/\delta\right)  ,$$
 where $\eta=\eta(2,z_4,\beta)$ as in Lemma \ref{integral} and $\omega$ is the exponent of the logarithm in Lemma \ref{lm_p=6}, that is
 $\omega=1$ if $\beta\neq2/5$ and $\omega=2$ if $\beta=2/5$.
 \end{proof}
 
 \proof(of the Theorems \ref{thm_d=2} and \ref{thm_d>2}) By Remark \ref{r1}, one has
 \begin{align*}
 & \left\{  \int_{\mathbb{R}}  \int_{\mathbb{T}^{d}}| r^{-{(d-1)}/{2}}\mathcal{D}\left(  \Omega,r,x\right)  |^{p}dxd\mu(r-R)\right\}  ^{1/p}\\
  =& \left\{  \int_{\mathbb{R}}  \int_{\mathbb{T}^{d}}| r^{-{(d-1)}/{2}}\mathcal{D}_{\delta}\left(  \Omega,r\pm\delta,x\right)  |^{p}dxd\mu(r-R)\right\}  ^{1/p}\\
  &+C\left\{  \int_{\mathbb{R}}  \int_{\mathbb{T}^{d}}| r^{(d-1)/2}\delta  |^{p}dxd\mu(r-R)\right\}  ^{1/p}\\
  \end{align*}
 If  $\delta=R^{-\left(  d-1\right)  /2}$ the last term is bounded. On the other hand, by Lemma \ref{Asymptotic Discrepancy}
  \begin{align*}
  & \left\{  \int_{\mathbb{R}}  \int_{\mathbb{T}^{d}}| r^{-{(d-1)}/{2}}\mathcal{D}_{\delta}\left(  \Omega,r\pm\delta,x\right)  |^{p}dxd\mu(r-R)\right\}  ^{1/p}\\
  \le& \left\{  \int_{\mathbb{R}}  \int_{\mathbb{T}^{d}}| \Phi(\delta,(d+1)/2,r\pm\delta,x)  |^{p}dxd\mu(r-R)\right\}  ^{1/p}\\
 & +   \left\{  \int_{\mathbb{R}}  \int_{\mathbb{T}^{d}}|\mathcal{R}_h(\delta,r\pm\delta,x)  |^{p}dxd\mu(r-R)\right\}  ^{1/p}.\\
    \end{align*}
    
 If $h\geq\left(  d-3\right)  /2$ then the last term is bounded, while the first term can be written as
 \[
 \left\{  \int_{\mathbb{R}}  \int_{\mathbb{T}^{d}}| \Phi(\delta,(d+1)/2,r,x)  |^{p}dxd\mu(r-(R\pm\delta))\right\}  ^{1/p}.
 \]
The theorem now follows from the two previous lemmas, with $R$ replaced by $R\pm\delta$.
 \endproof
 
 \section{The case of the ellipse}
 Here we assume $d=2$ and $\Omega=E=\{x\in\mathbb R^2:|M^{-1}x|\le 1\}$, where $M$ is a non singular 
 $2\times 2$ matrix. In this case the support function is $g(x)=|M^Tx|$. 
By the change of variable $M^Tx=y$ applied to all variables $n_1,\ldots,n_N,\,m_1,\ldots,m_N$, in this case Lemma \ref{lm_6_integral_ellipsoid} can be restated as
\begin{align*}
& \int_{\mathbb{R}} \int_{\mathbb{T}^{2}}  |\Phi(\delta,z,r,x)|^{2N}dxd\mu(r-R) \\
&\leq C\int_{\mathbb{R}^{2}} (1+\delta|k|)^{-\lambda}\\
&\times\int\limits_{\substack{m_1,\ldots,m_{N}\in\mathbb{R}^{2}\\ |m_1|, \ldots,|m_{N}| >1\\ m_1+\dots+m_{N}=k }} (1+\delta|m_1|)^{-\lambda}\dots(1+\delta|m_N|)^{-\lambda}|m_1|^{-\mathrm{Re}(z)}\dots|m_{N}|^{-\mathrm{Re}(z)}\\
&\times \int\limits_{\substack{n_1,\ldots,n_{N}\in\mathbb{R}^{2}\\ |n_1|, \ldots,|n_{N}| >1\\ n_1+\dots+n_{N}=k }}(1+\delta|n_1|)^{-\lambda}\dots(1+\delta|n_N|)^{-\lambda} |n_1|^{-\mathrm{Re}(z)}\dots|n_{N}|^{-\mathrm{Re}(z)} \\
&\times \left(1+| |m_1|+\dots+|m_{N}|-|n_1|-\dots-|n_{N}||\right)^{-\beta}\\
&\times d\sigma(n_1,\dots,n_{N}) d\sigma(m_1,\dots,m_{N}) dk.
\end{align*}

\begin{lemma}\label{lm_crucial} 
\begin{itemize}
\item[(1)] If $\beta\ge0$ and $0<2-\alpha<\beta$, there exists $C$ such that for every $-\infty<X<+\infty$,
$$\int_{0}^{+\infty}\left(  1+| X-\tau| \right)  ^{-\beta}\tau^{1-\alpha}d\tau\leq
C(1+|X|)^{2-\alpha-\min\{\beta,1\}}\log^{\delta_{1,\beta}}(2+|X|).
$$
\item[(2)] If $\alpha<2$, then for every $\beta$ there exists $C$ such that for every $-\infty<X<+\infty$,
$$\int_{0}^{1}\left(  1+| X-\tau| \right)  ^{-\beta}\tau^{1-\alpha}(-\log\left(\tau\right))  d\tau\leq C\left(  1+| X| \right)  ^{-\beta}.$$
\item[(3)] If $0\leq\beta<1$ and $2-\alpha\ge\beta$, then there exists $C$ such that for every $-\infty< X<+\infty$ and $2\leq T<+\infty$,
$$\int_{0}^{T}\left(  1+| X-\tau| \right)  ^{-\beta}\tau^{1-\alpha}d\tau\leq CT^{2-\alpha-\beta}\log^{\delta_{2-\alpha,\beta}}\left(  T\right)  .$$
\item[(4)] If $0\leq\beta<1$ and $\alpha>1$, there exists $C$ such that for every $-\infty< X<+\infty$ and $2\leq T<+\infty$,
$$\int_{T}^{+\infty}\left(  1+| X-\tau| \right)  ^{-\beta}\tau^{1-2\alpha}d\tau\leq CT^{2-2\alpha-\beta}.$$
\end{itemize}
\end{lemma}

\proof (1) If $X\leq0$, then
\begin{align*}
&\int_{0}^{+\infty}\left(  1+| X-\tau| \right)  ^{-\beta}\tau^{1-\alpha}d\tau\leq\left(  1+| X| \right)  ^{-\beta}\int_{0}^{1+| X| }\tau^{1-\alpha}d\tau+\int_{1+| X| }^{+\infty}\tau^{1-\alpha-\beta}d\tau\\
&\leq C\left(  1+| X| \right)^{2-\alpha-\beta}.
\end{align*}
If $0\leq X\leq1$, then
$$\int_{0}^{+\infty}\left(  1+| X-\tau| \right)  ^{-\beta}\tau^{1-\alpha}d\tau\leq\int_{0}^{2}\tau^{1-\alpha}d\tau+\int_{2}^{+\infty}\tau^{1-\alpha-\beta}d\tau\leq C.$$
If $X\geq1$, then
\begin{align*}
&\int_{0}^{+\infty}\left(  1+| X-\tau| \right)  ^{-\beta}\tau^{1-\alpha}d\tau\\
&  \leq CX^{-\beta}\int_{0}^{X/2}\tau^{1-\alpha}d\tau+C  X^{1-\alpha}\int_{X/2}^{2X}\left(  1+| X-\tau| \right)  ^{-\beta}d\tau+C\int_{2X}^{+\infty}\tau^{1-\alpha-\beta}d\tau\\
&  \leq
\begin{cases}
CX^{2-\alpha-\beta} & \text{if }0\leq\beta<1,\\
CX^{1-\alpha}\log\left(  1+X\right)  & \text{if }\beta=1,\\
CX^{1-\alpha} & \text{if }\beta>1.
\end{cases}
\end{align*}
(2) It suffices to observe that there exists $C$ such that for every $X$,
$$\max_{0\leq\tau\leq1}\left\{  \left(  1+| X-\tau| \right)^{-\beta}\right\}  \leq C\left(  1+| X| \right)  ^{-\beta}.$$
(3) Let $X\ge0$. If $T\leq X/2$, then
$$\int_{0}^{T}\left(  1+| X-\tau| \right)  ^{-\beta}\tau^{1-\alpha}d\tau
\leq C X^{-\beta} \int_{0}^{T}\tau^{1-\alpha}d\tau\leq C  X^{-\beta}{T}^{2-\alpha}
.$$
If $X/2\leq T\leq2X$, then
\begin{align*}
&\int_{0}^{T}\left(  1+| X-\tau| \right)  ^{-\beta}\tau^{1-\alpha}d\tau\\
&\leq CX^{-\beta}\int_{0}^{X/2}\tau^{1-\alpha}d\tau+CX^{1-\alpha}\int_{X/2}^{2X}\left(  1+| X-\tau| \right)  ^{-\beta}d\tau\\
&\leq CX^{2-\alpha-\beta}.
\end{align*}
If $T\geq2X$, then
\begin{align*}
&\int_{0}^{T}\left(  1+| X-\tau| \right)  ^{-\beta}\tau^{1-\alpha}d\tau\\
&  \leq CX^{-\beta}\int_{0}^{X/2}\tau^{1-\alpha}d\tau+CX^{1-\alpha}\int_{X/2}^{2X}\left(  1+| X-\tau| \right)  ^{-\beta}d\tau+C\int_{2X}^{T}\tau^{1-\alpha-\beta}d\tau\\
&  \leq CT^{2-\alpha-\beta}\log^{\delta_{2-\alpha,\beta}}\left(  T\right)  .
\end{align*}
If $X<0$ then simply observe that
\[
\int_{0}^{T}\left(  1+| X-\tau| \right)  ^{-\beta}\tau^{1-\alpha}d\tau
\leq
\int_{0}^{T}\left(  1+| |X|-\tau| \right)  ^{-\beta}\tau^{1-\alpha}d\tau.
\]
(4) As before, it suffices to show the result for $X\ge0$. If $T\leq X/2$, then
\begin{align*}
&\int_{T}^{+\infty}\left(  1+| X-\tau| \right)  ^{-\beta}\tau^{1-2\alpha}d\tau\\
&  \leq CX^{-\beta}\int_{T}^{X/2}\tau^{1-2\alpha}d\tau+CX^{1-2\alpha}\int_{X/2}^{2X}\left(  1+| X-\tau| \right)  ^{-\beta}d\tau+C\int_{2X}^{+\infty}\tau^{1-2\alpha-\beta}d\tau\\
&  \leq CX^{-\beta}T^{2-2\alpha}+CX^{2-2\alpha-\beta}\leq CT^{2-2\alpha-\beta}.
\end{align*}
If $X/2\leq T\leq2X$, then
\begin{align*}
&\int_{T}^{+\infty}\left(  1+| X-\tau| \right)  ^{-\beta}\tau^{1-2\alpha}d\tau\\
&  \leq T^{1-2\alpha}\int_{X/2}^{2X}\left(  1+| X-\tau| \right)  ^{-\beta}d\tau+C\int_{2X}^{+\infty}\tau^{1-2\alpha-\beta}d\tau\leq CT^{2-2\alpha-\beta}.
\end{align*}
If $T\geq2X$, then
$$\int_{T}^{+\infty}\left(  1+| X-\tau| \right)  ^{-\beta}\tau^{1-2\alpha}d\tau\leq C\int_{T}^{+\infty}\tau^{1-2\alpha-\beta}d\tau\leq CT^{2-2\alpha-\beta}.$$
\endproof

\begin{lemma}
\label{LM_Integrale}
(1) Let $3/2\leq\alpha<2$ and $\beta>2-\alpha$.
Then there exists $C$\ such that for every
$k\in\mathbb{R}^{2} $ with $|k|\ge2$ and for every $-\infty<Y<+\infty$,
\begin{align*}
&
{\int_{\mathbb{R}^{2}}}
\left\vert x\right\vert ^{-\alpha}\left\vert k-x\right\vert ^{-\alpha}\left(
1+\left\vert Y-|  x|  -|  k-x| \right\vert \right)
^{-\beta}dx\\
&  \leq C
|k|^{-\alpha}\log^{\delta_{3/2,\alpha}}(|k|)(1+|Y-|k||)^{2-\alpha-\min\{1,\beta\}}\log^{\delta_{1,\beta}}(2+|Y-|k||).
\end{align*}
(2) Let $3/2<\alpha<2$ and $0\leq\beta\leq2-\alpha$. 
Then there exists $C$\ such that for every
$k\in\mathbb{R}^{2} $ with $|k|\ge2$ and for every $-\infty<Y<+\infty$,
\begin{align*}
{\int_{\mathbb{R}^{2}}}
\left\vert x\right\vert ^{-\alpha}\left\vert k-x\right\vert ^{-\alpha}\left(
1+\left\vert Y-|  x|  -|  k-x| \right\vert \right)
^{-\beta}dx  \leq C
|k|^{2-2\alpha-\beta}\log^{\delta_{2-\alpha,\beta}}(|k|).
\end{align*}
(3) Let $\alpha=3/2$ and $\beta=1/2$. 
Then there exists $C$\ such that for every
$k\in\mathbb{R}^{2} $ with $|k|\ge2$ and for every $-\infty<Y<+\infty$,
\begin{align*}
{\int_{\mathbb{R}^{2}}}
\left\vert x\right\vert ^{-\alpha}\left\vert k-x\right\vert ^{-\alpha}\left(
1+\left\vert Y-|  x|  -|  k-x| \right\vert \right)
^{-\beta}dx  \leq C
|k|^{-\alpha}\log^2(|k|).
\end{align*}
(4) Let $3/4<\alpha<3/2$ and $\beta\geq0$. 
Then there exists $C$\ such that for every
$k\in\mathbb{R}^{2} $ with $|k|\ge2$ and for every $-\infty<Y<+\infty$,
\begin{align*}
&{\int_{\mathbb{R}^{2}}}
\left\vert x\right\vert ^{-\alpha}\left\vert k-x\right\vert ^{-\alpha}\left(
1+\left\vert Y-|  x|  -|  k-x| \right\vert \right)
^{-\beta}dx  \\
\leq& 
\begin{cases}
C|k|^{2-2\alpha-\beta} & \text { if } 0\leq \beta \leq 1/2,\\
C|k|^{-2\alpha+3/2}(1+|Y-|k||)^{1/2-\min\{1,\beta\}}\log^{\delta_{1,\beta}}(2+|Y-|k||) & \text { if } 1/2< \beta.\\
\end{cases}
\end{align*}
(5) Let $\alpha>1$. Then for every $k\in \mathbb R^2$
\[
\int_{|x|,|x-k|>1}|x|^{-\alpha}|x-k|^{-\alpha}dx\le(2\alpha-2)^{-1}.
\]
\end{lemma}

\proof The symmetry between $0$ and $k$ gives
\begin{align*}
& \int_{\mathbb{R}^{2}} | x| ^{-\alpha}| k-x| ^{-\alpha}\left(1+\left| Y-| x| -| k-x| \right|\right)  ^{-\beta}dx\\
&= 2\int_{\left\{  | x| +|k-x| \leq3| k| ,\ | x|\leq| k-x| \right\}  } | x| ^{-\alpha}| k-x| ^{-\alpha}\left(1+\left| Y-| x| -| k-x| \right|\right)  ^{-\beta}dx\\
&+ 2\int_{  | x| +|k-x| \geq3| k| ,\ | x|\leq| k-x|   }| x| ^{-\alpha}| k-x| ^{-\alpha}\left(1+\left| Y-| x| -| k-x| \right|\right)  ^{-\beta}dx\\
&\leq C| k| ^{-\alpha} \int_{  | x| +|k-x| \leq3| k|   }| x| ^{-\alpha}\left(  1+\left| Y-|x| -| k-x| \right| \right)  ^{-\beta}dx\\
&+ C\int_{ | x| +|k-x| \geq3| k|   }| x| ^{-2\alpha}\left(  1+\left| Y-|x| -| k-x| \right| \right)  ^{-\beta}dx.
\end{align*}
We estimate here the first integral, the second being studied similarly.
The integral is invariant under rotations of $k$, so that one can assume $k=\left(  | k| ,0\right)  $. 
Write in polar coordinates $y=\left(  \rho\cos\left(  \vartheta\right)  ,\rho\sin\left(\vartheta\right) \right)  $, with $0\leq\rho<+\infty$, $0\leq\vartheta\leq2\pi$.
In these polar coordinates the ellipse $\left\{| x| +| k-x| =\tau\right\}  $ has equation
$$\rho=\frac{\tau^{2}-| k| ^{2}}{2\left(  \tau-|k| \cos\left(  \vartheta\right)  \right)  }.$$
In the variables $\left(  \tau,\vartheta\right)  $, $|k| \leq\tau<+\infty$, $0\leq\vartheta\leq2\pi$, one has
$$\frac{d\rho}{d\tau}  =\frac{\tau^{2}-2| k| \tau\cos\left(  \vartheta\right)  +| k| ^{2}}{2\left(\tau-| k| \cos\left(  \vartheta\right)  \right)  ^{2}},$$
and
\begin{align*}
dx    =\rho  d\rho d\vartheta =  \frac{\tau^{2}-| k| ^{2}}{2\left(\tau-| k| \cos\left(  \vartheta\right) \right)}  \frac{\tau^{2}-2| k| \tau\cos\left(  \vartheta\right)  +| k| ^{2}}{2\left(  \tau-| k|\cos\left(  \vartheta\right)  \right)  ^{2}}  d\tau d\vartheta.
\end{align*}
Hence,
\begin{align*}
& \int_{  | x| +|k-x| \leq3| k|   }\left(  1+\left| Y-| x| -| k-x|\right| \right)  ^{-\beta}| x| ^{-\alpha}dx\\
&= \int_{| k| }^{3| k| }\int_{0}^{2\pi}\left(  1+| Y-\tau| \right)  ^{-\beta}\left(  \frac{\tau^{2}-| k| ^{2}}{2\left(  \tau-| k|\cos\left(  \vartheta\right)  \right)  }\right)  ^{1-\alpha}
 \frac{\tau^{2}-2| k| \tau\cos\left(  \vartheta\right)  +| k| ^{2}}{2\left(  \tau-| k|\cos\left(  \vartheta\right)  \right)  ^{2}} d\tau d\vartheta \\
&= 2^{\alpha-2}\int_{| k| }^{3| k| }\left(  1+| Y-\tau| \right)  ^{-\beta}\left(  \tau-|k| \right)  ^{1-\alpha}\left(  1+\left(  | k|/\tau\right)  \right)  ^{1-\alpha}\\
&  \times\int_{0}^{2\pi}\left(  1-2\left(  | k| /\tau\right)  \cos\left(\vartheta\right)  +\left(  | k| /\tau\right)  ^{2}\right)\left(  1-\left(  | k| /\tau\right)  \cos\left(\vartheta\right)  \right)  ^{\alpha-3}d\vartheta d\tau.
\end{align*}
The term $1+\left(  | k| /\tau\right)  $ in the last double integral is bounded between 1 and 2. Hence,
\begin{align*}
&\int_{ | x| +|k-x| \leq3| k|  }\left(  1+\left| Y-| x| -| k-x|\right| \right)  ^{-\beta}| x| ^{-\alpha}dx\\
&  \leq C\int_{0}^{2| k| }\left(  1+\left| Y-| k| -\tau\right| \right)^{-\beta}\tau^{1-\alpha}E\left(  | k| /\left(  |k| +\tau\right)  ,\alpha\right)  d\tau,
\end{align*}
where
\[
E(t,\alpha)=2 \int_{0}^{\pi}\left(  1-2t\cos\left(  \vartheta\right)  +t^{2}\right)\left(  1-t\cos\left(  \vartheta\right)  \right)  ^{\alpha-3} d\vartheta.
\]
When $0<t<1/2$, $E(t,\alpha)\leq C$. 
When $1/2\leq t<1,$ the integral over $\pi/2\leq\vartheta\leq\pi$ is bounded independently of $t$, and when $0\leq\vartheta\leq\pi/2$ one has  $1-\vartheta^{2}/2\leq\cos\left(  \vartheta\right)  \leq1-4\vartheta^{2}/\pi^{2}$. 
Hence one ends up with the integral

\begin{align*}
&\int_{0}^{\pi/2}\left(  1-2t\left(  1-\vartheta^{2}/2\right)  +t^{2}\right)  \left(  1-t\left(  1-4\vartheta^{2}/\pi^{2}\right)\right)  ^{\alpha-3}d\vartheta\\
&  =\int_{0}^{\pi/2}\left(  \left(  1-t\right)  ^{2}+t\vartheta^{2}\right)\left(  1-t+4t\vartheta^{2}/\pi^{2}\right)  ^{\alpha-3}d\vartheta\\
&  \leq C\left(  1-t\right)  ^{\alpha-1}\int_{0}^{1-t}d\vartheta+C\left(  1-t\right)  ^{\alpha-3}\int_{1-t}^{\sqrt{1-t}}\vartheta^{2}d\vartheta+C\int_{\sqrt{1-t}}^{\pi/2}\vartheta^{2\alpha-4}d\vartheta\\
&  \leq
\begin{cases}
C(1-t)^{\alpha-3/2} & \textit{if }\alpha<3 /2,\\
-C\log\left(  1-t\right) & \textit{if }\alpha=3/2,\\
C & \textit{if }\alpha>3  /2.
\end{cases}
\end{align*}

Hence
\begin{align*}
&| k| ^{-\alpha} \int_{ | x| +|k-x| \leq3| k|  }| x| ^{-\alpha}\left(  1+\left| Y-|x| -| k-x| \right| \right)  ^{-\beta}dx\\
\leq & C| k| ^{-\alpha-\min\{0,\alpha-3/2\}}\\
&\times \int_{0}^{2| k| }\left(  1+\left| Y-| k| -\tau\right| \right)^{-\beta}\tau^{1-\alpha+\min\{0,\alpha-3/2\}}\log^{\delta_{\alpha,3/2}}((|k|+\tau)/\tau)d\tau.
\end{align*}

Similarly,
\begin{align*}
& \int_{  | x| +|k-x| \geq3| k|   }| x| ^{-2\alpha}\left(  1+\left| Y-|x| -| k-x| \right| \right)  ^{-\beta}dx\\
\leq & C\int_{2| k| }^{+\infty}\left(  1+\left| Y-| k| -\tau\right| \right)^{-\beta}\tau^{1-2\alpha}d\tau.
\end{align*}

Assume $3/2<\alpha<2$ and $\alpha>2-\beta$. Then applying Lemma \ref{lm_crucial} (1), we obtain
\begin{align*}
&
| k| ^{-\alpha}\int_{0}^{2| k| }\left(  1+\left| Y-| k| -\tau\right| \right)^{-\beta}\tau^{1-\alpha}d\tau
+ \int_{2| k| }^{+\infty}\left(  1+\left| Y-| k| -\tau\right| \right)^{-\beta}\tau^{1-2\alpha}d\tau\\
&\leq | k| ^{-\alpha}\int_{0}^{+\infty }\left(  1+\left| Y-| k| -\tau\right| \right)^{-\beta}\tau^{1-\alpha}d\tau\\
&\leq C |k|^{-\alpha}(1+|Y-|k||)^{2-\alpha-\min\{1,\beta\}}\log^{\delta_{1,\beta}}(2+|Y-|k||).
\end{align*}

Assume now $\alpha=3/2$ and $\beta>1/2$. Then 
\begin{align*}
&
| k| ^{-\alpha}\int_{0}^{2| k| }\left(  1+\left| Y-| k| -\tau\right| \right)^{-\beta}\tau^{1-\alpha}\log\left((|k|+\tau)/\tau\right)d\tau\\
&+ \int_{2| k| }^{+\infty}\left(  1+\left| Y-| k| -\tau\right| \right)^{-\beta}\tau^{1-2\alpha}d\tau\\
&\leq | k| ^{-\alpha}\int_{0}^{+\infty }\left(  1+\left| Y-| k| -\tau\right| \right)^{-\beta}\tau^{1-\alpha}\log\left((|k|+\tau)/\tau\right)d\tau\\
&\leq | k| ^{-\alpha}\log\left(1+|k|\right)\int_{0}^{+\infty }\left(  1+\left| Y-| k| -\tau\right| \right)^{-\beta}\tau^{1-\alpha}d\tau\\
&+| k| ^{-\alpha}\int_{0}^{+\infty }\left(  1+\left| Y-| k| -\tau\right| \right)^{-\beta}\tau^{1-\alpha}\log\left(1+1/\tau\right)d\tau.
\end{align*}
Applying Lemma \ref{lm_crucial} (1), we obtain
\begin{align*}
&| k| ^{-\alpha}\log\left(|k|\right)\int_{0}^{+\infty }\left(  1+\left| Y-| k| -\tau\right| \right)^{-\beta}\tau^{1-\alpha}d\tau\\
&\leq C| k| ^{-\alpha}\log\left(|k|\right)(1+|Y-|k||)^{2-\alpha-\min\{\beta,1\}}\log^{\delta_{\beta,1}}(2+|Y-|k||).
\end{align*}
Applying Lemma \ref{lm_crucial} (1) and (2), we obtain
\begin{align*}
&| k| ^{-\alpha}\int_{0}^{+\infty }\left(  1+\left| Y-| k| -\tau\right| \right)^{-\beta}\tau^{1-\alpha}\log\left(1+1/\tau\right)d\tau\\
&\leq | k| ^{-\alpha}\int_{0}^{+\infty }\left(  1+\left| Y-| k| -\tau\right| \right)^{-\beta}\tau^{1-\alpha}d\tau\\
& +| k| ^{-\alpha}\int_{0}^{1 }\left(  1+\left| Y-| k| -\tau\right| \right)^{-\beta}\tau^{1-\alpha}(-\log\left(\tau\right))d\tau\\
&\leq C| k| ^{-\alpha}(1+|Y-|k||)^{2-\alpha-\min\{\beta,1\}}\log^{\delta_{\beta,1}}(2+|Y-|k||).
\end{align*}

If $3/2<\alpha<2$ and $\alpha\leq 2-\beta$, then applying Lemma \ref{lm_crucial} (3), we obtain
\begin{align*}
| k| ^{-\alpha}\int_{0}^{2| k| }\left(  1+\left| Y-| k| -\tau\right| \right)^{-\beta}\tau^{1-\alpha}d\tau\leq C | k| ^{2-2\alpha-\beta}\log^{\delta_{2-\alpha,\beta}}(|k|),
\end{align*}
and, by Lemma \ref{lm_crucial} (4),
\begin{align*}
\int_{2| k| }^{+\infty}\left(  1+\left| Y-| k| -\tau\right| \right)^{-\beta}\tau^{1-2\alpha}d\tau\leq C | k| ^{2-2\alpha-\beta}.
\end{align*}

Assume $\alpha=3/2$ and $\beta=1/2$. Then, by Lemma \ref{lm_crucial} (2), (3) and (4),
\begin{align*}
&
| k| ^{-\alpha}\int_{0}^{2| k| }\left(  1+\left| Y-| k| -\tau\right| \right)^{-\beta}\tau^{1-\alpha}\log\left((|k|+\tau)/\tau\right)d\tau\\
&+ \int_{2| k| }^{+\infty}\left(  1+\left| Y-| k| -\tau\right| \right)^{-\beta}\tau^{1-2\alpha}d\tau\\
&\leq C| k| ^{-\alpha}\log(|k|)\int_{0}^{2| k| }\left(  1+\left| Y-| k| -\tau\right| \right)^{-\beta}\tau^{1-\alpha}d\tau\\
&+C| k| ^{-\alpha}\int_{0}^{1 }\left(  1+\left| Y-| k| -\tau\right| \right)^{-\beta}\tau^{1-\alpha}(-\log\left(\tau\right))d\tau\\
&+ \int_{2| k| }^{+\infty}\left(  1+\left| Y-| k| -\tau\right| \right)^{-\beta}\tau^{1-2\alpha}d\tau\\
&\leq C|k|^{-\alpha}\log^2(|k|).
\end{align*}

Assume $3/4<\alpha< 3/2$ and $\beta>1/2$. Then applying Lemma \ref{lm_crucial} (1), we obtain
\begin{align*}
&
| k| ^{-2\alpha+3/2}\int_{0}^{2| k| }\left(  1+\left| Y-| k| -\tau\right| \right)^{-\beta}\tau^{-1/2}d\tau\\
&+ \int_{2| k| }^{+\infty}\left(  1+\left| Y-| k| -\tau\right| \right)^{-\beta}\tau^{1-2\alpha}d\tau\\
&\leq | k| ^{-2\alpha+3/2}\int_{0}^{+\infty }\left(  1+\left| Y-| k| -\tau\right| \right)^{-\beta}\tau^{-1/2}d\tau\\
&\leq C |k|^{-2\alpha+3/2}(1+|Y-|k||)^{1/2-\min\{1,\beta\}}\log^{\delta_{1,\beta}}(2+|Y-|k||).
\end{align*}

If $3/4<\alpha<3/2$ and $0\leq\beta\leq 1/2$, then applying Lemma \ref{lm_crucial} (3), we obtain
\begin{align*}
| k| ^{-2\alpha+3/2}\int_{0}^{2| k| }\left(  1+\left| Y-| k| -\tau\right| \right)^{-\beta}\tau^{-1/2}d\tau\leq C | k| ^{2-2\alpha-\beta},
\end{align*}
and, by Lemma \ref{lm_crucial} (4),
\begin{align*}
\int_{2| k| }^{+\infty}\left(  1+\left| Y-| k| -\tau\right| \right)^{-\beta}\tau^{1-2\alpha}d\tau\leq C | k| ^{2-2\alpha-\beta}.
\end{align*}

Finally,  (5) follows by a simple rearrangement inequality (see \cite[Theorem 3.4]{LL}),
\[
\int_{|x|,|x-k|>1}|x|^{-\alpha}|x-k|^{-\alpha}dx
\le \int_{|x|>1}|x|^{-2\alpha}dx=(2\alpha-2)^{-1}.
\] 
\endproof

The next two lemmas are the counterpart of Lemmas \ref{lm_p=4} and \ref{lm_p=6} in the case of the ellipse.
The minimal value $z_4$ of $\mathrm{Re}(z)$ for which the $L^4$ norm of $\Phi(\delta,z,\cdot,\cdot)$
is proved to be bounded is lowered from $\max\{(6-\beta)/4,11/8\}$ to $\max\{(6-\beta)/4,5/4\}$. Similarly, $z_6$
is lowered from $\max\{(10-\beta)/6,8/5\}$ to $\max\{(10-\beta)/6,3/2\}$.

\begin{figure}[h!]
\begin{center}
\includegraphics[scale=0.6]{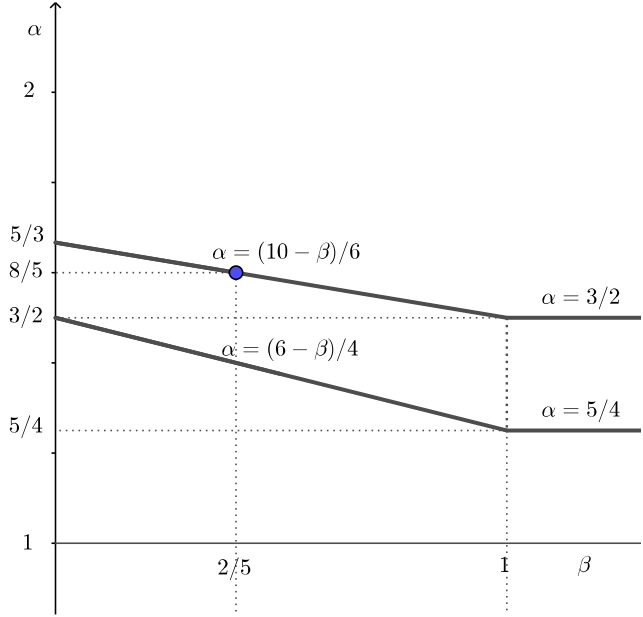}
\end{center}
\caption{The minimal values of $\alpha=\mathrm{Re}(z)$ for the ellipse in dimension $2$. The values $z_4$ (bottom) and $z_6$ (top). Compare with Figure \ref{F3}{\footnotesize(A)}.}
\end{figure}

\begin{lemma} \label{lm_p=4_disc} 
 Let $d=2$ and $\beta\ge 0$. Set $z_4=\max\{(6-\beta)/4,5/4\}$.
If $\mathrm{Re}(z)  \geq z_4$, then there exists $C>0$ such that for every $R\geq 1$ and $0<\delta<1/2$,
\begin{align*}
  \int_{\mathbb{R}}  \int_{\mathbb{T}^{2}}|\Phi\left(  \delta,z,r,x\right)  | ^{4}dxd\mu(r-R)
  \leq
\begin{cases}
C & \text{if }\mathrm{Re}\left(  z\right)  >z_4,\\
C\log\left(  1/\delta\right)  & \text{if }\mathrm{Re}\left(  z\right)=z_4 \text{ and }
\beta\neq 1,\\
C\log^2\left(  1/\delta\right)  & \text{if }\mathrm{Re}\left(  z\right)=z_4 \text { and }
\beta= 1.
\end{cases}
\end{align*}
 \end{lemma}
 
 \proof
Call $\alpha=\mathrm{Re}\left(  z\right)$. By Lemma \ref{lm_6_integral_ellipsoid} with $N=2$, it suffices to estimate
 \begin{align*}
 &  \int_{\mathbb{R}^{2}} \left(  1+\delta| k| \right)^{-\lambda} \int_{|m| ,| k-m| >1}\left(  1+\delta| m| \right)^{-\lambda}|m|^{-\alpha}|k-m|^{-\alpha}\\
 & \times \int_{|n|,| k-n| >1}\left(  1+\delta| n| \right)^{-\lambda}|n|^{-\alpha}|k-n|^{-\alpha}\\
  &\times\left(  1+| |m| +|k-m| -|n|-|k-n| |\right)  ^{-\beta}dndm dk.
 \end{align*}
 Split $\mathbb{R}^{2}$ as $\left\{|k| \leq2\right\}  \cup\left\{  | k|
 \geq2\right\}  $. By Lemma \ref{LM_Integrale} (5), the integral over the disc $\left\{ | k| \leq2\right\}  $ is bounded by
 \[
   \int_{  | k| \leq2  }\left(\int_{|m|,|k-m|>1}| m| ^{-\alpha}| k-m| ^{-\alpha}dm\right)  ^{2}dk
 \leq C.
 \]
 Consider now the case $\left\{  | k| \geq2\right\}  $. Assume $\beta>1/2$. Apply Lemma \ref{LM_Integrale} (4) to the integral with respect to $n$ with 
 $Y$ replaced with $|m|+|k-m|
 $. 
  \begin{align*}
 & \int_{|n|,| k-n| >1}|n|^{-\alpha}|k-n|^{-\alpha} \left(  1+| |m| +|k-m| -|n|-|k-n| |\right)  ^{-\beta}dn\\
 &  \leq C|k|^{3/2-2\alpha}(1+||m|+|k-m|-|k||)^{1/2-\min\{1,\beta\}}\\
 &\times\log^{\delta_{1,\beta}}(2+||m|+|k-m|-|k||).
 \end{align*}
 Thus we obtain the integral
  \begin{align*}
 &|k|^{3/2-2\alpha} \int_{|m|,| k-m| >1}|m|^{-\alpha}|k-m|^{-\alpha} (1+||m|+|k-m|-|k||)^{1/2-\min\{1,\beta\}}\\
 &\times\log^{\delta_{1,\beta}}(2+||m|+|k-m|-|k||)dm.
 \end{align*}
If $\beta=1$ and $\alpha>5/4$, the logarithm can be removed as long as one replaces $\alpha$ with a slightly smaller value greater than $5/4$. Therefore, if $(\alpha,\beta)\neq(5/4,1)$, then we only need to estimate
 \begin{align*}
 &|k|^{3/2-2\alpha} \int_{|m|,| k-m| >1}|m|^{-\alpha}|k-m|^{-\alpha} (1+||m|+|k-m|-|k||)^{1/2-\min\{1,\beta\}}dm\\
 &\leq C|k|^{4-4\alpha-\min\{1,\beta\}}
 \end{align*}   
  where we have applied Lemma \ref{LM_Integrale} (4).
 Finally, the integral over $\left\{  | k| \geq2\right\}  $ gives
 \begin{align*}
 &\int_{\substack{  | k| \geq2}  }(1+\delta|k|)^{-\lambda}|k| ^{4-4\alpha-\min\{1,\beta\}}dk\\
 &\leq 
\begin{cases}
C & \text{ if } \alpha>z_4,\\
C\log(1/\delta) & \text{ if } \alpha=z_4 \text{ and } \beta\neq1.
\end{cases}
\end{align*}
Assume $0\leq\beta\leq1/2$. Apply Lemma \ref{LM_Integrale}(4) to the integral with respect to $n$ with 
 $Y$ replaced with $|m|+|k-m|
 $. 
  \begin{align*}
  \underset{|n|,| k-n| >1}\int|n|^{-\alpha}|k-n|^{-\alpha} \left(  1+| |m| +|k-m| -|n|-|k-n| |\right)  ^{-\beta}dn \leq C|k|^{2-2\alpha-\beta}.
 \end{align*}
 Thus we obtain the integral
  \begin{align*}
 |k|^{2-2\alpha-\beta} \int_{|m|,| k-m| >1}|m|^{-\alpha}|k-m|^{-\alpha}dm
 \leq C|k|^{4-4\alpha-\beta}.
 \end{align*}
 Finally, the integral over $\left\{  | k| \geq2\right\}  $ gives
 \begin{align*}
 \int_{\substack{  | k| \geq2}  }(1+\delta|k|)^{-\lambda}|k| ^{4-4\alpha-\beta}dk
 \leq 
\begin{cases}
C  & \text{ if } \alpha>z_4\\
C\log(1/\delta) & \text{ if } \alpha=z_4.
\end{cases}
\end{align*}

It remains to study the case $\alpha=5/4$ and $\beta=1$. By Lemma \ref{LM_Integrale} (4),
 \begin{align*}
 &  \int_{\mathbb{R}^{2}} \left(  1+\delta| k| \right)^{-\lambda} \int_{|m| ,| k-m| >1}\left(  1+\delta| m| \right)^{-\lambda}|m|^{-\alpha}|k-m|^{-\alpha}\\
 & \times \int_{|n|,| k-n| >1}|n|^{-\alpha}|k-n|^{-\alpha} \left(  1+| |m| +|k-m| -|n|-|k-n| |\right)  ^{-\beta}dndm dk\\
 &\leq
 C\int_{\mathbb{R}^{2}} \left(  1+\delta| k| \right)^{-\lambda} |k|^{-1}\int_{|m| ,| k-m| >1}\left(  1+\delta| m| \right)^{-\lambda}|m|^{-\alpha}|k-m|^{-\alpha}\\
 &  \times(1+||m|+|k-m|-|k||)^{-1/2}\log(2+||m|+|k-m|-|k||)dmdk.
 \end{align*}

 Since there is a $C>0$ such that for $0<\delta<1/2$ 
 \[
(1+\delta|m|)^{-\lambda}\log(2+2|m|)\leq C\log(1/\delta),
  \]
 and  $2+||m|+|k-m|-|k||\leq 2+2|m|$, we obtain 
  \begin{align*}
 &\int_{\mathbb{R}^{2}} \left(  1+\delta| k| \right)^{-\lambda} |k|^{-1}\int_{|m| ,| k-m| >1}\left(  1+\delta| m| \right)^{-\lambda}|m|^{-\alpha}|k-m|^{-\alpha}\\
 &  \times(1+||m|+|k-m|-|k||)^{-1/2}\log(2+||m|+|k-m|-|k||)dmdk\\
 &  \leq C\log(1/\delta)\int_{\mathbb{R}^{2}} \left(  1+\delta| k| \right)^{-\lambda} |k|^{-1}\int_{|m| ,| k-m| >1}|m|^{-\alpha}|k-m|^{-\alpha}\\
 &  \times(1+||m|+|k-m|-|k||)^{-1/2}dmdk\\
 &  \leq C\log(1/\delta)\int_{\mathbb{R}^{2}} \left(  1+\delta| k| \right)^{-\lambda} |k|^{-2}dk\leq C\log^2(1/\delta),
 \end{align*}
 by Lemma \ref{LM_Integrale} (4).
   \endproof

\begin{lemma} \label{lm_p=6_disc} 
 Let $d=2$ and $\beta\ge 0$. Set $z_6=\max\{(10-\beta)/6,3/2\}$.
If $\mathrm{Re}(z)  \geq z_6$, then there exists $C>0$ such that for every $R\geq 1$ and $0<\delta<1/2$,
\begin{align*}
  \int_{\mathbb{R}}  \int_{\mathbb{T}^{2}}|\Phi\left(  \delta,z,r,x\right)  | ^{6}dxd\mu(r-R)
  \leq
\begin{cases}
C & \text{if }\mathrm{Re}\left(  z\right)  >z_6,\\
C\log\left(  1/\delta\right)  & \text{if }\mathrm{Re}\left(  z\right)=z_6 \text{ and }
0\leq\beta<2/5,\\
C\log^2\left(  1/\delta\right)  & \text{if }\mathrm{Re}\left(  z\right)=z_6 \text { and }
\beta= 2/5,\\
C\log\left(  1/\delta\right)  & \text{if }\mathrm{Re}\left(  z\right)=z_6 \text{ and }
2/5<\beta<1,\\
C\log^5\left(  1/\delta\right)  & \text{if }\mathrm{Re}\left(  z\right)=z_6 \text{ and }
\beta=1,\\
C\log^4\left(  1/\delta\right)  & \text{if }\mathrm{Re}\left(  z\right)=z_6 \text{ and }
\beta>1.
\end{cases}
\end{align*}
 \end{lemma}
 
 \proof
Call $\alpha=\mathrm{Re}\left(  z\right)$. By Lemma \ref{lm_6_integral_ellipsoid} with $N=3$, it suffices to estimate
 \begin{align*}
 &  \int_{\mathbb{R}^{2}} \left(  1+\delta| k| \right)^{-3\lambda} \\
 &\times\iint\limits_{\substack{|m_1| ,|m_2|,\\|k-m_1-m_2| >1}}\left(  1+\delta|m_1| \right)^{-\lambda}\left(  1+\delta| m_2| \right)^{-\lambda}|m_1|^{-\alpha}|m_2|^{-\alpha}|k-m_1-m_2|^{-\alpha}\\
 & \times \iint\limits_{\substack{|n_1| ,|n_2|,\\|k-n_1-n_2| >1}}(1+\delta|n_1|)^{-2\lambda}|n_1|^{-\alpha}|n_2|^{-\alpha}|k-n_1-n_2|^{-\alpha}\\
 &  \times\left(  1+| |m_1| +|m_2|+|k-m_1-m_2| -|n_1|-|n_2|-|k-n_1-n_2| |\right)  ^{-\beta}\\
 & \times dn_2 dn_1 dm_2 dm_1 dk.
 \end{align*}
Split $\mathbb{R}^{2}$ as $\left\{|k| \leq2\right\}  \cup\left\{  | k|
 \geq2\right\}  $. The integral over the disc $\left\{ | k| \leq2\right\}  $ is bounded by
 \begin{align*}
 &  \int_{  | k| \leq2  }\left(\int_{|m_1|>1}| m_1| ^{-\alpha}\int_{\mathbb{R}^{2}}| m_2| ^{-\alpha}| k-m_1-m_2| ^{-\alpha}dm_2dm_1\right)  ^{2}dk\\
 &  =C\int_{  | k| \leq2 }\left(\int_{|m_1|>1}| m_1| ^{-\alpha}| k-m_1| ^{2-2\alpha}dm_1\right)  ^{2}dk\\
 &  =C\int_{ | k| \leq1/2  }\left(\int_{|m_1|>1}| m_1| ^{2-3\alpha}dm_1\right)  ^{2}dk
 +C\int_{ 1/2\leq | k| \leq2  }|k| ^{8-6\alpha}dk\leq C,
 \end{align*}
 since $\alpha\ge z_6>4/3$.
 
 Consider now the case $\left\{  | k| \geq2\right\}  $. Assume $\alpha>2-\beta$ and $\alpha>3/2$. When $|k-n_1|\geq 2$, apply Lemma \ref{LM_Integrale} (1) 
 to the integral with respect to $n_2$ with $k$ replaced with $k-n_1$ and
 $Y$ replaced with $|m_1|+|m_2|+|k-m_1-m_2|-|n_1|$, and when $|k-n_1|\leq2$ apply Lemma \ref{LM_Integrale} (5) to the same integral,
  \begin{align*}
 & \underset{|n_1| ,|n_2|,| k-n_1-n_2| >1}{\int\int}|n_1|^{-\alpha}|n_2|^{-\alpha}|k-n_1-n_2|^{-\alpha}\\
&\times \left(  1+| |m_1| +|m_2|+|k-m_1-m_2| -|n_1|-|n_2|-|k-n_1-n_2| |\right)  ^{-\beta}
  dn_2 dn_1\\
 &  \leq C\int_{\mathbb R^2}|n_1|^{-\alpha}|k-n_1|^{-\alpha}\\
& \times\left(  1+ ||m_1| +|m_2|+|k-m_1-m_2| -|n_1|-|k-n_1| |\right)  ^{2-\alpha-\min\{1,\beta\}}\\
&\times \log^{\delta_{\beta,1}}(2+ ||m_1| +|m_2|+|k-m_1-m_2| -|n_1|-|k-n_1| |)
  dn_1\\
  &+C\int_{|k-n_1|<2}|n_1|^{-\alpha}dn_1.  
 \end{align*}
 The last integral is bounded by a constant times $|k|^{-\alpha}$.
 If $\beta=1$, the logarithm can be removed as long as one replaces $\alpha$ with a slightly smaller value greater than $3/2$. Thus we only need to estimate
 \begin{align*}
 &  \int_{\mathbb R^2}|n_1|^{-\alpha}|k-n_1|^{-\alpha}\\
& \times\left(  1+ ||m_1| +|m_2|+|k-m_1-m_2| -|n_1|-|k-n_1| |\right)  ^{2-\alpha-\min\{1,\beta\}}
  dn_1\\
   &  \leq C|k|^{-\alpha}|k|^{\max\{0,4-2\alpha-\min\{1,\beta\}\}}\log^{\delta_{\alpha,2-\beta/2}}(|k|),
 \end{align*}
 where we have applied Lemma \ref{LM_Integrale} (1) and (2).
 Moreover,
 \begin{align*}
 &\int_{\mathbb{R}^{2}}| m_1| ^{-\alpha}\int_{\mathbb{R}^{2}}| m_2| ^{-\alpha}| k-m_1-m_2| ^{-\alpha}dm_2dm_1\\
 &=C\int_{\mathbb{R}^{2}}| m_1| ^{-\alpha}|k-m_1| ^{2-2\alpha}dm_1=C| k| ^{4-3\alpha}.
 \end{align*}
 Finally, the integral over $\left\{  | k| \geq2\right\}  $ gives
 \begin{align*}
 &\int_{\substack{  | k| \geq2}  }(1+\delta|k|)^{-3\lambda}|k| ^{4-4\alpha+\max\{0,4-2\alpha-\min\{1,\beta\}\}}\log^{\delta_{\alpha,2-\beta/2}}(|k|)dk\\
 &\leq 
\begin{cases}
C & \text{ if } \alpha>(10-\beta)/6,\\
C\log(1/\delta) & \text{ if } \alpha=(10-\beta)/6 \text{ and } 2/5<\beta<1.
\end{cases}
\end{align*}

Assume now $3/2<\alpha\leq2-\beta$ so that $0\leq\beta<1/2$. Apply Lemma \ref{LM_Integrale} (2) and (5) to the integral with respect to $n_2$ with $k$ replaced with $k-n_1$,
  \begin{align*}
 & \underset{|n_1| ,|n_2|,| k-n_1-n_2| >1}{\int\int}|n_1|^{-\alpha}|n_2|^{-\alpha}|k-n_1-n_2|^{-\alpha}\\
&\times \left(  1+| |m_1| +|m_2|+|k-m_1-m_2| -|n_1|-|n_2|-|k-n_1-n_2| |\right)  ^{-\beta}
  dn_2 dn_1\\
 &  \leq C\underset{\mathbb R^2}{\int}|n_1|^{-\alpha}|k-n_1|^{2-2\alpha-\beta} \log^{\delta_{2-\alpha,\beta}}(2+|k-n_1|)\,dn_1\\
 & \leq C |k|^{4-3\alpha-\beta}\log^{\delta_{2-\alpha,\beta}}(|k|).
 \end{align*}
 Moreover,
 \begin{align*}
 &\int_{\mathbb{R}^{2}}| m_1| ^{-\alpha}\int_{\mathbb{R}^{2}}| m_2| ^{-\alpha}| k-m_1-m_2| ^{-\alpha}dm_2dm_1\\
 &=C\int_{\mathbb{R}^{2}}| m_1| ^{-\alpha}|k-m_1| ^{2-2\alpha}dm_1=C| k| ^{4-3\alpha}.
 \end{align*}
 Finally, the integral over $\left\{  | k| \geq2\right\}  $ gives
 \begin{align*}
 &\int_{\substack{  | k| \geq2}  }(1+\delta|k|)^{-3\lambda}|k| ^{8-6\alpha-\beta}\log^{\delta_{\alpha,2-\beta}}(|k|)dk\\
 &\leq 
\begin{cases}
C & \text{ if } \alpha>(10-\beta)/6,\\
C\log(1/\delta) & \text{ if } \alpha=(10-\beta)/6 \text{ and } 0\leq\beta<2/5,\\
C\log^2(1/\delta) & \text{ if } \alpha=(10-\beta)/6 \text{ and } \beta=2/5.
\end{cases}
\end{align*}
It remains to study the case $\alpha=3/2$ and $\beta\ge1$. 
Apply Lemma \ref{LM_Integrale} (1) and (5) to the integral with respect to $n_2$ with $k$ replaced with $k-n_1$,
  \begin{align*}
 &  \left(  1+\delta| k| \right)^{-3\lambda} 
 \left(  1+\delta|m_1| \right)^{-\lambda}\left(  1+\delta| m_2| \right)^{-\lambda}|m_1|^{-\alpha}|m_2|^{-\alpha}|k-m_1-m_2|^{-\alpha}\\
 & \times \underset{|n_1| ,|n_2|,| k-n_1-n_2| >1}{\int\int}(1+\delta|n_1|)^{-2\lambda}|n_1|^{-\alpha}|n_2|^{-\alpha}|k-n_1-n_2|^{-\alpha}\\
 &  \times\left(  1+| |m_1| +|m_2|+|k-m_1-m_2| -|n_1|-|n_2|-|k-n_1-n_2| |\right)  ^{-\beta} dn_2 dn_1 \\
 &  \leq C(1+\delta|k|)^{-3\lambda}
 \left(  1+\delta|m_1| \right)^{-\lambda}\left(  1+\delta| m_2| \right)^{-\lambda}|m_1|^{-\alpha}|m_2|^{-\alpha}|k-m_1-m_2|^{-\alpha}\\
 &\times\int_{\mathbb R^2}(1+\delta|n_1|)^{-2\lambda}|n_1|^{-\alpha}|k-n_1|^{-\alpha}\log(2+|k-n_1|)\\
& \times\left(  1+ ||m_1| +|m_2|+|k-m_1-m_2| -|n_1|-|k-n_1| |\right)  ^{-1/2}\\
&\times \log^{\delta_{\beta,1}}(2+ ||m_1| +|m_2|+|k-m_1-m_2| -|n_1|-|k-n_1| |)
  dn_1 \\
  &+ C(1+\delta|k|)^{-3\lambda}
|m_1|^{-\alpha}|m_2|^{-\alpha}|k-m_1-m_2|^{-\alpha}
\int_{|n_1-k|<2}|n_1|^{-\alpha}
  dn_1
 \end{align*}
 Since there is a $C>0$ such that for $0<\delta<1/2$
 \begin{align*}
 &(1+\delta|k|)^{-\lambda}(1+\delta|m_1|)^{-\lambda}(1+\delta|m_2|)^{-\lambda} (1+\delta|n_1|)^{-\lambda}\\&\times\log(2+|k|+|m_1|+|m_2|+|n_1|)\\
 &\leq C\log(1/\delta),
  \end{align*}
 we obtain 
  \begin{align*}
&C\log^{1+\delta_{1,\beta}}(1/\delta) (1+\delta|k|)^{-\lambda}
|m_1|^{-\alpha}|m_2|^{-\alpha}|k-m_1-m_2|^{-\alpha}\\
 &\times\int_{\mathbb R^2}|n_1|^{-\alpha}|k-n_1|^{-\alpha}\\
 &\times\left(  1+ ||m_1| +|m_2|+|k-m_1-m_2| -|n_1|-|k-n_1| |\right)  ^{-1/2}
  dn_1 \\
  &+ C(1+\delta|k|)^{-3\lambda}
|m_1|^{-\alpha}|m_2|^{-\alpha}|k-m_1-m_2|^{-\alpha}\int_{|n_1-k|<2}|n_1|^{-\alpha}
  dn_1.
  \end{align*}
By Lemma \ref{LM_Integrale}(3), the above integrals are bounded by
 \begin{align*}
&C\log^{1+\delta_{1,\beta}}(1/\delta) (1+\delta|k|)^{-\lambda}|k|^{-\alpha}\log^2{|k|}
 |m_1|^{-\alpha}|m_2|^{-\alpha}|k-m_1-m_2|^{-\alpha}.
  \end{align*}

 Moreover,
 \begin{align*}
 &\int_{\mathbb{R}^{2}}| m_1| ^{-\alpha}\int_{\mathbb{R}^{2}}| m_2| ^{-\alpha}| k-m_1-m_2| ^{-\alpha}dm_2dm_1\\
 &=C\int_{\mathbb{R}^{2}}| m_1| ^{-\alpha}|k-m_1| ^{2-2\alpha}dm_1=C| k| ^{4-3\alpha}.
 \end{align*}
 Finally, the integral over $\left\{  | k| \geq2\right\}  $ gives
 \begin{align*}
 \log^{1+\delta_{1,\beta}}(1/\delta)\int_{\substack{  | k| \geq2}  }(1+\delta|k|)^{-\lambda}|k| ^{4-4\alpha}\log^2{|k|}dk
 \leq 
C\log^{4+\delta_{1,\beta}}(1/\delta).
\end{align*}
  \endproof

 \begin{lemma} \label{lm_Interpolation_sphere_d=2} 
The notation is as in the previous lemmas and let $d=2$. 

\noindent If $0\leq\beta<2/5$ then there exist a constant $C$ such that for every  $\mathrm{Re}\left(  z\right)  \geq 3/2$,  for every $R\ge 1$ and $0<\delta<1/2$,
 \begin{align*}
   \left\{  \int_{\mathbb{R}}  \int_{\mathbb{T}^{2}}| \Phi\left(  \delta,z,r,x\right)  | ^{p}dxd\mu(r-R)\right\}^{1/p}
   \leq
 \begin{cases}
 C & \text{if }p<4+2\beta,\\
  C\log^{1/p}\left(  1/\delta\right)  & \text{if }p=4+2\beta.
 \end{cases}
 \end{align*}
 If $\beta=2/5$ then there exist a constant $C$ such that for every  $\mathrm{Re}\left(  z\right)  \geq 3/2$,  for every $R\ge 1$ and $0<\delta<1/2$,
 \begin{align*}
   \left\{  \int_{\mathbb{R}}  \int_{\mathbb{T}^{2}}| \Phi\left(  \delta,z,r,x\right)  | ^{p}dxd\mu(r-R)\right\}^{1/p}
   \leq
 \begin{cases}
 C & \text{if }p<4+2\beta,\\
  C\log^{1/p+1/12}\left(  1/\delta\right)  & \text{if }p=4+2\beta.
 \end{cases}
 \end{align*}
 If $2/5<\beta<1$ then there exist a constant $C$ such that for every  $\mathrm{Re}\left(  z\right)  \geq 3/2$,  for every $R\ge 1$ and $0<\delta<1/2$,
 \begin{align*}
  \left\{  \int_{\mathbb{R}}  \int_{\mathbb{T}^{2}}| \Phi\left(  \delta,z,r,x\right)  | ^{p}dxd\mu(r-R)\right\}^{1/p}
  \leq
 \begin{cases}
 C & \text{if }p<4+2\beta,\\
  C\log^{1/p}\left(  1/\delta\right)  & \text{if }p=4+2\beta.
 \end{cases}
 \end{align*}
  If $\beta=1$ then there exist a constant $C$ such that for every  $\mathrm{Re}\left(  z\right)  \geq 3/2$,  for every $R\ge 1$ and $0<\delta<1/2$,
 \begin{align*}
  \left\{  \int_{\mathbb{R}}  \int_{\mathbb{T}^{2}}| \Phi\left(  \delta,z,r,x\right)  | ^{p}dxd\mu(r-R)\right\}^{1/p}
  \leq
 \begin{cases}
 C & \text{if }p<6,\\
  C\log^{1/p+2/3}\left(  1/\delta\right)  & \text{if }p=6.
 \end{cases}
 \end{align*}
  If $\beta>1$ then there exist a constant $C$ such that for every  $\mathrm{Re}\left(  z\right)  \geq 3/2$,  for every $R\ge 1$ and $0<\delta<1/2$,
 \begin{align*}
  \left\{  \int_{\mathbb{R}}  \int_{\mathbb{T}^{2}}| \Phi\left(  \delta,z,r,x\right)  | ^{p}dxd\mu(r-R)\right\}^{1/p}
  \leq
 \begin{cases}
 C & \text{if }p<6,\\
  C\log^{1/p+1/2}\left(  1/\delta\right)  & \text{if }p=6.
 \end{cases}
 \end{align*}

 \end{lemma}
 
 \begin{proof}
 Again, it is enough to prove the result for $z=3/2$. The case $\beta=0$ is contained in Lemma \ref{lm_p=4_disc}. If $\beta>0$
 the proof follows by complex interpolation with $a=4$, $b=6$, $A=z_4+\varepsilon$, $B=z_6+\varepsilon$, with $\varepsilon\geq0$. The norms $H$ and $K$ are given in Lemma \ref{lm_p=4_disc} and Lemma \ref{lm_p=6_disc}. 
 Set
 $$\frac{3}{2}=\left(  1-\vartheta\right)  A+\vartheta B.$$
 This gives
 $$
 \vartheta=
 \begin{cases}
 \frac{3(\beta-4\varepsilon)}{2+\beta} \text { if } \beta<1,\\
 1-4\varepsilon \text { if } \beta\geq1,
 \end{cases}
 $$
 and
 $$\frac{1}{p}=\frac{\left(  1-\vartheta\right)  }{a}+\frac{\vartheta}{b}=
 \begin{cases}
 \frac{1+2\varepsilon}{2(2+\beta)} & \text{ if } \beta<1,\\
 \frac{1+2\varepsilon}{6} & \text{ if } \beta\geq1.
 \end{cases}
 $$

 When $\varepsilon>0$ and $p<\min\{6,4+2\beta\} $,
 $$\left\{  \int_{\mathbb{R}} \int_{\mathbb{T}^{2}}| \Phi\left(  \delta,3/2,x\right)  |^{p}dxd\mu(r-R)\right\}  ^{1/p}\leq C.$$
 When $\varepsilon=0$ and $p=\min\{6,4+2\beta\}  $,
 $$\left\{  \int_{\mathbb{R}}  \int_{\mathbb{T}^{d}}| \Phi\left(  \delta,3  /2,r,x\right)  |^{p}dxd\mu(r-R)\right\}  ^{1/p}\leq C\log^{1/p+\eta(1-\vartheta)/4+\omega\vartheta/6}\left(  1/\delta\right)  ,$$
 where $\omega=4$ if $\beta=1$, $\omega=3$ if $\beta>1$, $\omega=1$ if $\beta=2/5$ and $\omega=0$ in the other cases, while
 $\eta=1$ if $\beta=1$ and $\eta=0$ in the other cases.
 \end{proof}

The proof of Theorem \ref{thm_d=2_ellipse} can now be concluded as in the case of the general convex set with smooth boundary  with strictly positive curvature.

\bibliographystyle{amsplain}

\end{document}